\documentclass[pdftex]{article}

\usepackage{amsthm,amsmath,amssymb}
\usepackage{proof}

\usepackage{tikz}

\newtheorem{theorem}{Theorem}[section]
\newtheorem{lemma}[theorem]{Lemma}

\newtheorem{proposition}[theorem]{Proposition}
\newtheorem{remark}[theorem]{Remark}
\newtheorem{definition}[theorem]{Definition}
\newtheorem{example}[theorem]{Example}

\newcommand{\noi}{\noindent}

\newcommand{\T}{\mbox{\sf true}}
\newcommand{\F}{\mbox{\sf false}} 

\newcommand{\AND}{\mbox{$\land$}} 
\newcommand{\OR}{\mbox{$\lor$}} 
\newcommand{\IMP}{\mbox{$\to$}} 
\newcommand{\NOT}{\mbox{$\lnot$}} 
\newcommand{\BOT}{\mbox{$\bot$}}

\newcommand{\GA}{{\mit \Gamma}}
\newcommand{\DE}{{\mit \Delta}}
\newcommand{\SI}{{\mit \Sigma}}
\newcommand{\PI}{{\mit \Pi}}
\newcommand{\PH}{{\mit \Phi}}
\newcommand{\PS}{{\mit \Psi}}

\newcommand{\R}{\prec}
\newcommand{\RI}{\succ}

\newcommand{\SEQ}[2]{\mbox{${#1}\Rightarrow{#2}$}}
\newcommand{\SSEQ}[2]{\mbox{${#1}\Rrightarrow{#2}$}}
\newcommand{\Darrow}{\Rightarrow\!\!\!\!\Rightarrow}
\newcommand{\DSEQ}[2]{\mbox{${#1}\Darrow{#2}$}}
\newcommand{\XSEQ}[2]{\mbox{${#1}\blacktriangleright{#2}$}}

\newcommand{\HGL}{\ensuremath{{\bf GL}_{\rm H}}}
\newcommand{\HS}{\ensuremath{{\bf S}_{\rm H}}}
\newcommand{\HD}{\ensuremath{{\bf D}_{\rm H}}}
\newcommand{\HDa}{\ensuremath{{\bf D}_{\rm H}^{2}}}
\newcommand{\HDb}{\ensuremath{{\bf D}_{\rm H}^{3}}}

\newcommand{\HDi}{\ensuremath{{\bf D}_{\rm H}^{i}}}
\newcommand{\HDL}[1]{\ensuremath{{\bf D}_{\rm H}[#1]}}
\newcommand{\HDaL}[1]{\ensuremath{{\bf D}_{\rm H}^{2}[#1]}}
\newcommand{\HDbL}[1]{\ensuremath{{\bf D}_{\rm H}^{3}[#1]}}

\newcommand{\LK}{\ensuremath{{\bf LK}_{\Rightarrow}}}
\newcommand{\LKS}{\ensuremath{{\bf LK}_{\Rrightarrow}}}
\newcommand{\LKD}{\ensuremath{{\bf LK}_{\Darrow}}}
\newcommand{\sGL}{\ensuremath{{\bf GL}_{\rm seq}}}
\newcommand{\sS}{\ensuremath{{\bf S}_{\rm seq}}}
\newcommand{\sDa}{\ensuremath{{\bf D}_{\rm seq}^{2}}}
\newcommand{\sDb}{\ensuremath{{\bf D}_{\rm seq}^{3}}}
\newcommand{\sDi}{\ensuremath{{\bf D}_{\rm seq}^{i}}}

\newcommand{\Kushida}{\ensuremath{{\bf S}_{\rm Kushida}}}

\newcommand{\GLLINF}{\ensuremath{{\bf LinGL}^{\rm F}}}
\newcommand{\GLLIN}{\ensuremath{{\bf LinGL}}}

\newcommand{\DLIN}{\ensuremath{\bf LinD}}

\newcommand{\DLINO}[1]{\ensuremath {{\bf LinD}^{#1}}}
\newcommand{\HGLLIN}{\ensuremath{{\bf LinGL}_{\rm H}}}

\newcommand{\RuleGL}{\ensuremath{\text{GL$\Box$}}}
\newcommand{\RuleS}{\ensuremath{\text{GLtoS}}}
\newcommand{\RuleLeftBox}{\ensuremath{\text{S$\Box$left}}}
\newcommand{\RuleDa}{\ensuremath{\text{GLtoD}}}
\newcommand{\RuleDb}{\ensuremath{\text{StoD}}}

\newcommand{\KushidaGLL}{\ensuremath{\text{Kushida-GLtoS$\Box$left}}}

\newcommand{\KushidaGLR}{\ensuremath{\text{Kushida-GL$\Box$}}}
\newcommand{\KushidaSR}{\ensuremath{\text{Kushida-GLtoS$\Box$}}}

\newcommand{\SF}[1]{\mbox{\rm Sub}(#1)}

\newcommand{\Kmodel}[3]{\langle #1, #2, #3\rangle}
\newcommand{\Kframe}[2]{\langle #1, #2\rangle}

\newcommand{\PV}{{\bf PropVar}}

\newcommand{\DEF}{\stackrel{\text{def}}{\Longleftrightarrow}}

\newcommand{\FORMULA}{{\bf Form}}
\newcommand{\MP}[1]{\ensuremath{{\rm{MP}}[#1]}}
\newcommand{\SL}[1]{\ensuremath{{\bf S}[#1]}}
\newcommand{\DL}[1]{\ensuremath{{\bf D}[#1]}}
\newcommand{\DaL}[1]{\ensuremath{{\bf D}^{2}[#1]}}
\newcommand{\DbL}[1]{\ensuremath{{\bf D}^{3}[#1]}}

\begin{document}

\title{Cut-free sequent calculi for the provability logic D}
\author{
Ryo Kashima\thanks{
Department of Mathematical and Computing Science, 
Institute of Science Tokyo, Japan.
{\tt kashima@is.titech.ac.jp}
}
\and Taishi Kurahashi\thanks{
Graduate School of System Informatics, Kobe University, Japan.
{\tt kurahashi@people.kobe-u.ac.jp}
}
\and Sohei Iwata\thanks{
Division of Liberal Arts and Sciences, Aichi-Gakuin University, Japan.
{\tt siwata@dpc.agu.ac.jp}
}
\and So Morioka\thanks{
{\tt soubaseoct@icloud.com}
}
}
\date{\today}

\maketitle

\abstract{
We say that a Kripke model is a GL-model
if the accessibility relation $\prec$ 
is transitive and converse well-founded.
We say that a Kripke model is a D-model
if it is obtained by
attaching infinitely many worlds 
$t_1, t_2, \ldots$, and $t_\omega$
to a world $t_0$ of a GL-model
so that 
$t_0 \succ t_1 \succ t_2 \succ \cdots \succ t_\omega$.
A non-normal modal logic {\bf D},
which was studied by Beklemishev \cite{Beklemishev99},
is characterized as follows.
A formula $\varphi$ is a theorem of {\bf D}
if and only if 
$\varphi$ is true at $t_\omega$ in any D-model.
{\bf D} is an intermediate logic between
the provability logics {\bf GL} and {\bf S}.
A Hilbert-style proof system for {\bf D} is known,
but there has been no sequent calculus.
In this paper, we establish two sequent calculi for {\bf D},
and show the cut-elimination theorem.
We also introduce new Hilbert-style systems for {\bf D}
by interpreting the sequent calculi.
Moreover,
we show that 
D-models can be defined 
using an arbitrary limit ordinal as well as $\omega$.
Finally,
we show a general result as follows.
Let $X$ and $X^+$ be arbitrary modal logics.
If the relationship between semantics of $X$ and semantics of $X^+$
is equal to that of {\bf GL} and {\bf D},
then $X^+$ can be axiomatized based on $X$
in the same way as the new axiomatization of {\bf D} based on {\bf GL}.
}


\section{Introduction}
\label{sec:intro}

We say that a Kripke model is a {\em GL-model}
if the accessibility relation $\prec$ 
is transitive and converse well-founded.
Then, a {\em D-model} 
is obtained by
attaching infinitely many worlds 
$t_1, t_2, \ldots$ (called `tail'), and $t_\omega$ (called `bottom')
to a world $t_0$ of a GL-model
so that 
$t_0 \succ t_1 \succ t_2 \succ \cdots \succ t_\omega$
(Figure~\ref{Fig:D-model}, Left).
A non-normal modal logic {\bf D},
which was studied by Beklemishev \cite{Beklemishev99},
is characterized as follows.
A formula $\varphi$ is a theorem of {\bf D}
if and only if 
$\varphi$ is true at the bottom of any D-model.\footnote{
Following \cite{Beklemishev99}, we use the name {\bf D},
while {\bf D} denotes the modal logic of serial frames in most literature.
}

\begin{figure}[h]

\begin{minipage}{7cm}
\begin{tikzpicture} [x=1mm,y=1mm]

\draw(0,9) ellipse (14 and 6.5);

\node at (0,5) {$\bullet$};
\node at (3, 4) {$t_0$};
\draw [->, thick] (-1,6)--(-3,8);
\draw [->, thick] (0,6)--(0,9);
\draw [->, thick] (1,6)--(3,8);

\draw [->, thick] (0,2)--(0,4);

\node at (0, 1) {$\bullet$};
\node at (3, 0) {$t_1$};
\draw [->,thick] (0,-2)--(0,0);

\node at (0,-3) {$\bullet$};
\node at (3,-4) {$t_2$};
\draw [->,thick] (0,-6)--(0,-4);

\node at (0,-7) {$\cdot$};
\node at (0,-8) {$\cdot$};
\node at (0,-9) {$\cdot$};

\node at (0,-11) {$\bullet$};
\node at (3,-12) {$t_\omega$};

\node [right] at (5,14) {\colorbox{white}{GL-model}};
\node [right] at (5,-3) {\colorbox{white}{tail}};
\node [right] at (4,-11) {\colorbox{white}{bottom}};
\draw [thick,dotted] (4,1) to [out=330, in=30] (4,-9);

\node at (-40,0) {};
\node at (3,-23) {};
\end{tikzpicture} 
\end{minipage}
\begin{minipage}{7cm}
\begin{tikzpicture} [x=1mm,y=1mm]

\draw(0,9) ellipse (14 and 6.5);

\node at (0,5) {$\bullet$};
\node at (3, 4) {$t_0$};
\draw [->, thick] (-1,6)--(-3,8);
\draw [->, thick] (0,6)--(0,9);
\draw [->, thick] (1,6)--(3,8);

\draw [->, thick] (0,2)--(0,4);

\node at (0, 1) {$\bullet$};
\node at (3, 0) {$t_1$};
\draw [->,thick] (0,-2)--(0,0);

\node at (0,-3) {$\bullet$};
\node at (3,-4) {$t_2$};
\draw [->,thick] (0,-6)--(0,-4);

\node at (0,-7) {$\cdot$};
\node at (0,-8) {$\cdot$};
\node at (0,-9) {$\cdot$};

\node at (0,-11) {$\bullet$};
\node at (3,-12) {$t_\omega$};

\draw [->,thick] (0,-14)--(0,-12);
\node at (0,-15) {$\bullet$};
\node at (5,-16) {$t_{\omega+1}$};

\draw [->,thick] (0,-18)--(0,-16);
\node at (0,-19) {$\cdot$};
\node at (0,-20) {$\cdot$};
\node at (0,-21) {$\cdot$};

\node at (0,-23) {$\bullet$};
\node at (3,-23) {$t_{\lambda}$};

\node [right] at (5,14) {\colorbox{white}{GL-model}};
\node [right] at (9,-9) {\colorbox{white}{tail}};
\node [right] at (4,-23) {\colorbox{white}{bottom}};
\draw [thick,dotted] (5,1) to [out=330, in=30] (5,-20);

\node at (-40,0) {};
\end{tikzpicture} 
\end{minipage}
\caption{(Left) D-model by Beklemishev. (Right) New D-model.}
\label{Fig:D-model}
\end{figure}

{\bf D} is a provability logic as follows.
A formula $\varphi$ is a theorem of {\bf D}
if and only if 
any $\varphi^*$ is true in the standard model of arithmetic 
where $\varphi^*$ is obtained from $\varphi$
by interpreting the modal operator $\Box$ as 
the provability predicate of
a c.e.~extension of Peano Arithmetic
that is $\Sigma_1$-sound but not sound.
In this paper, we do not argue {\bf D} from the perspective of provability logics,
but we consider {\bf D} as an interesting modal logic
which is non-normal (not closed under the necessitation rule)
and has simple Kripke-style semantics.
We establish sequent calculi, cut-elimination, and new Hilbert-style axiomatizations for {\bf D};
furthermore, we show that 
we can define D-models
using an arbitrary limit ordinal $\lambda$ as well as $\omega$
(Figure~\ref{Fig:D-model}, Right).

{\bf D} is closely related to the well-known provability logics 
{\bf GL} and {\bf S}
(see, e.g., \cite{Boolos, Solovay} for the basic results on 
{\bf GL} and {\bf S})\footnote{
Following \cite{Beklemishev99, CZ, Visser84}, we use the name {\bf S},
while it is called {\bf GLS} in \cite{Boolos, KK, Kushida},
${\rm G}^*$ in \cite{Boolos80},
and $G'$ in \cite{Solovay}.
}.
A Hilbert-style proof system for {\bf GL} is known
as ${\bf K} + \Box(\Box\varphi \IMP \varphi) \IMP \Box\varphi$.
A Hilbert-style proof system for {\bf S} (we call this system {\HS})
is as follows.
The axioms are all theorems of {\bf GL} and all formulas
$\Box\varphi \IMP \varphi$;
and sole inference rule is modus ponens.
Then a Hilbert-style proof system for {\bf D} (we call this system {\HD})
is known to be 
obtained from {\HS}
by restricting $\varphi$ in the axiom scheme $\Box\varphi \IMP \varphi$
to be $\BOT$ or $\Box\psi \OR \Box\sigma$
(see \cite{Beklemishev99}).
Therefore 
{\bf D} is an intermediate logic between {\bf GL} and {\bf S}. 

As for sequent calculi,
{\bf GL} has been  well studied 
(see, e.g., \cite{Gore} and its references), 
and 
{\bf S} has been studied in \cite{Beklemishev87, KK, Kushida}.
But there has been no sequent calculi for {\bf D}.
The sequent calculus for {\bf S} in \cite{Beklemishev87, KK, Kushida} 
was inspired by the Hilbert-style system {\HS};
so one may try to make a sequent calculus for {\bf D}
based on the system {\HD}.
However, this attempt does not seem to work well
because the axiom 
$\Box(\Box\psi \OR \Box\sigma) \IMP (\Box\psi \OR \Box\sigma)$
does not seem to be translatable into a rule of sequent calculi.

In this paper, we establish sequent calculi for {\bf D}
so that the completeness with respect to D-models naturally holds.
A key idea of our calculi is the use of three kinds of sequents.
While Kushida \cite{Kushida} used two kinds of sequents
to make the calculus for {\bf S},
we add one more kind.
We call the three kinds 
`GL-sequents',
`S-sequents',
and 
`D-sequents';
these respectively correspond to
the truth at the GL-model,
at the tail, and at the bottom
in Figure~\ref{Fig:D-model}.
Moreover, 
as the names suggest,
these correspond to the provability of 
{\bf GL}, {\bf S}, and {\bf D}, 
respectively.

Strictly speaking, we give two sequent calculi, called {\sDa} and {\sDb}.
The latter calculus {\sDb} is cut-eliminable,
and the former calculus {\sDa} is not fully cut-eliminable but has 
the subformula property (we call this property `analytic').
We show semantic cut-elimination for both calculi
(i.e., completeness of cut-free {\sDb} 
and analytic {\sDa})
and syntactic cut-elimination for {\sDb}.
These semantic arguments are extensions of that for {\bf S} in \cite{KK},
and the syntactic cut-elimination is reduced to that for {\bf S} by \cite{Kushida}.

\begin{remark}
\textup{
A proof in {\sDb} has three layers with three kinds of sequents,
and the syntactic cut-elimination for the bottom layer 
(with D-sequents) is reduced to that for the middle layer 
(with S-sequents).
Similar arguments --- a proof has layered structure
and the cut-elimination for the lower layer
is reduced to that for the upper layer ---
are found in \cite{Lang}.
}
\end{remark}

Then, two new Hilbert-style proof systems
(we call these {\HDa} and {\HDb})
for {\bf D} are  obtained 
by interpreting the sequent calculi {\sDa} and {\sDb} respectively.
Here, not only the existing system {\HD} but also the new systems {\HDa} and {\HDb}
have 
`all theorems of {\bf GL}' 
as their axioms;
so it is natural to argue a generalization as follows.
Let $L$ be an arbitrary modal logic,
and let {\HDL{L}}, {\HDaL{L}}, and {\HDbL{L}} be the 
Hilbert-style systems 
obtained from 
{\HD}, {\HDa}, and {\HDb}, respectively,
by replacing the axioms `theorems of {\bf GL}'
with `theorems of $L$'.
On the other hand,
let ${\cal F}$ be a class of Kripke frames, 
and let `${\bf D}[{\cal F}]$-model' denote
any Kripke model described as Figure~\ref{Fig:D-model}
in which `GL-model' is replaced with `${\cal F}$-model'.
Then we show the following.
\begin{quote}
If $L$ is characterized by ${\cal F}$
and if a certain condition is satisfied,
then the following conditions are equivalent.
(1) $\varphi$ is true at the bottom of any ${\bf D}[{\cal F}]$-model.
(2) $\varphi$ is a theorem of $\HDaL{L}$.
(3) $\varphi$ is a theorem of $\HDbL{L}$.
\end{quote}
This statement is just the completeness theorem of {\HDa} and {\HDb}
if $L = {\bf GL}$
and ${\cal F}$ is the class of transitive and converse well-founded frames.
It seems that 
the condition 
`$\varphi$ is a theorem of $\HDL{L}$'
is not equivalent to the above three conditions.
This fact shows that 
the new proof systems ---{\HDa}, {\HDb}, and the two sequent calculi---
well reflect the essence of the modal logic {\bf D},
and that
the new systems are more natural than the existing system {\HD}.

The structure of this paper is as follows.
In Section 2,
we present results that are known or will be shown in later sections.
In Section 3, 
we recall the sequent calculi for {\bf GL} and {\bf S},
and we introduce two sequent calculi {\sDa} and {\sDb}.
In Section 4, we give syntactic arguments on the sequent calculi.
In Section 5, we show the semantic cut-elimination.
In Section 6, we introduce 
Hilbert-style systems {\HDa} and {\HDb}.
In Section 7, we show the general result 
on $\HDaL{L}$ and $\HDbL{L}$.


\section{Preliminaries and results}

Formulas are constructed from 
propositional variables, propositional constant $\BOT$, 
logical operator $\IMP$,
and modal operator $\Box$.
The other operators are defined as abbreviations as usual.
The letters 
$p, q, \ldots$ denote propositional variables
and 
the letters 
$\varphi, \psi, \ldots$ denote formulas.
The set of propositional variables is called {\PV}
and the set of formulas is called {\FORMULA}.
Parentheses are omitted as, for example, 
$\Box\varphi \AND \psi \IMP \Box\pi \OR \sigma = 
((\Box\varphi) \AND \psi) \IMP ((\Box\pi) \OR \sigma)$.

Hilbert-style proof systems {\HGL}, {\HS}, and {\HD}
are known as follows, 
where the subscript {H} denotes `Hilbert-style'.
\begin{quote}
Axioms of {\HGL}: 
Tautologies, 
$\Box(\varphi \IMP \psi) \IMP (\Box\varphi \IMP \Box\psi)$,
and 
$\Box(\Box\varphi \IMP \varphi) \IMP \Box\varphi$.
\\
Rules of {\HGL}:
$
\infer[\text{(modus ponens)}]{\psi}{\varphi \IMP \psi  & \varphi}
\text{ and }
\infer[(\Box)]{\Box\varphi}{\varphi}.
$

Axioms of {\HS}: 
Theorems of {\HGL} and 
$\Box\varphi \IMP \varphi$.
\\
Rule of {\HS}:  Modus ponens.

Axioms of {\HD}: Theorems of \HGL, 
$\NOT\Box\BOT$,
and 
$\Box(\Box\varphi \OR \Box\psi) \IMP \Box\varphi \OR \Box\psi$.
\\
Rule of {\HD}: Modus ponens.
\end{quote}
The symbol $\vdash$ denotes  provability as usual.
We have 
$
({\HGL} \vdash \varphi)
 \Longrightarrow
({\HD} \vdash \varphi)
 \Longrightarrow
({\HS} \vdash \varphi)$, 
but the converse does not hold in general;
for example, 
$\Box p \IMP p$ 
is provable in {\HS}
but not provable in {\HD}.
Note that neither {\HD} nor {\HS} has the inference rule ($\Box$).
For example,
we have 
${\HD} \vdash \NOT\Box\BOT$
but
${\HD} \not\vdash \Box\NOT\Box\BOT$.

\medskip

A {\em Kripke model} $\Kmodel{W}{\R}{V}$  consists of 
a non-empty set $W$ of {\em worlds}, 
an {\em accessibility relation} 
${\R} \subseteq W \times W$,
and 
a {\em valuation} $V: W \times {\PV}  \to \{\T, \F\}$.
The domain of $V$ is extended to $W \times {\FORMULA}$ 
by the following.
\begin{quote}
$V(w,\BOT) = \F$.
\\
$V(w, \varphi\IMP\psi) = \T
\Longleftrightarrow
V(w, \varphi) = \F
\mbox{ or }
V(w, \psi) = \T.$
\\
$V(w,\Box\varphi) = \T
\Longleftrightarrow
(\forall w' \RI w)(V(w',\varphi) = \T).$
\end{quote}
$\Kframe{W}{\R}$ is called 
the {\em Kripke frame of} this model,
and $\Kmodel{W}{\R}{V}$ is called 
a Kripke model {\em based on} the frame $\Kframe{W}{\R}$.
We say that a formula $\varphi$ is {\em true at} a world $w$ if
$V(w,\varphi) = \T$.

A Kripke frame $\Kframe{W}{\R}$ is called a {\em GL-frame}
if ${\R}$ is converse well-founded
(i.e.,
there is no infinitely ascending sequence $x_1 \R x_2 \R x_3 \cdots$)
and transitive.
A Kripke model based on a GL-frame is called a {\em GL-model}. 
The completeness of {\HGL} is well-known as below (see, e.g., \cite{Boolos}).
\begin{proposition}[Completeness of {\HGL}]
\label{prop:PreliminaryGL}
The following are equivalent.
\begin{itemize}
\item 
$\HGL \vdash \varphi$.
\item 
$\varphi$ is true at any world of any GL-model.
\end{itemize}
\end{proposition}

\medskip

To describe the semantics for {\HS} and {\HD},
we need further definitions.

\begin{definition}
\textup{
Let $\lambda$ be a limit ordinal.
We say that a frame
$\Kframe{W^+}{\R^+}$
is a $\lambda$-extension 
of $\Kframe{W}{\R}$ 
if the following two conditions hold
for some $t_0 \in W$.
\begin{quote}
$W^+ = W \uplus \{t_\alpha \mid 0 < \alpha \leq \lambda\}$.
\\
${\R}^+ = {\R} 
\cup \{(t_{\alpha}, t_{\beta}) \mid 0 \leq \beta < \alpha \leq \lambda\} 
\cup \{(t_{\alpha}, x) \mid 0 < \alpha \leq \lambda \mbox{ and } (t_{0}, x) \in {\R}\}.$
\end{quote}
(Refer to Figure~\ref{Fig:D-model}, in which the ellipse is $\Kframe{W}{\R}$.)
The infinite set 
$\{t_\alpha \mid 0 < \alpha < \lambda\}$
is called the {\em tail} and 
the world $t_\lambda$ is called the {\em bottom} of this frame.
If a valuation $V^+$ coincides with $V$ for any worlds in $W$, 
then we say that 
the Kripke model $M^+ = \Kmodel{W^+}{\R^+}{V^+}$ is a 
{\em $\lambda$-extension} of $M = \Kmodel{W}{\R}{V}$. 
Moreover, 
$M^+$ is called 
a {\em constant} 
or 
a {\em strongly constant}
$\lambda$-extension of $M$
if the following condition holds, respectively.
\begin{quote}
\begin{tabular}{rl}
Constant: &
$(\forall p \in \PV)
(\forall \alpha, \beta < \lambda)
(V^+(t_\alpha, p) = V^+(t_\beta, p))$.
\\
Strongly constant: &
$(\forall p \in \PV)
(\forall \alpha, \beta \leq \lambda)
(V^+(t_\alpha, p) = V^+(t_\beta, p))$.
\end{tabular}
\end{quote}
We say that a formula $\varphi$ is 
{\em eventually always true in the tail of} $M^+$
if $(\exists \alpha<\lambda)(\alpha < \forall \beta < \lambda)(V(t_\beta, \varphi) = \T)$.
}
\end{definition}

The following proposition is easy to be proved.
This will be used implicitly in this paper.

\begin{proposition}
If $\Kmodel{W^+}{\R^+}{V^+}$ is a $\lambda$-extension of $\Kmodel{W}{\R}{V}$,
then for any world $w \in W$ and any formula $\varphi$,
we have $V^+(w, \varphi) = V(w, \varphi)$.
\end{proposition}

The completeness of {\HS} is described below,
which will be proved in Section~\ref{sec:cut-free-completeness}.
\begin{proposition}[Completeness of {\HS}]
\label{prop:PreliminaryS}
Let $\lambda$ be a limit ordinal.
The following are equivalent.
\begin{itemize}
\item [{\rm (s1)}]
$\HS \vdash \varphi$.
\item [{\rm (s2)}]
$\varphi$
is true at the bottom of any strongly constant $\lambda$-extension of any GL-model.
\item [{\rm (s3)}]
$\varphi$
is eventually always true in the tail of
any (strongly) constant $\lambda$-extension of any GL-model.
\item [{\rm (s4)}]
$\varphi$
is eventually always true in the tail of
any $\lambda$-extension of any GL-model.
\end{itemize}
\end{proposition}

\begin{remark}
\textup{
The completeness of {\HS} has been expressed in several different statements in \cite{Beklemishev99, Boolos80, CZ, KK, Visser84}, 
but all of them are essentially included in the $\lambda = \omega$ case of Proposition~\ref{prop:PreliminaryS}.
In this paper, we extend it to arbitrary limit ordinals.
}
\end{remark}

\begin{remark}
\label{rem:bound-lambda}
\textup{
The limit ordinal $\lambda$ is arbitrarily fixed 
at the beginning of
Proposition~\ref{prop:PreliminaryS}.
On the other hand,
we can consider propositions in which $\lambda$ is 
bound at each sentence;
for example, there are two 	variants of the condition (s2)
as follows.
\begin{quote}
(s2$^+$) 
$\varphi$
is true at the bottom of any strongly constant 
$\lambda'$-extension of any GL-model,
for {\em any} limit ordinal $\lambda'$.
\\
(s2$^-$) 
$\varphi$
is true at the bottom of any strongly constant 
$\lambda'$-extension of any GL-model,
for {\em some} limit ordinal $\lambda'$.
\end{quote}
The conditions (s2$^+$), (s2$^-$),  and (s2)
are equivalent for any $\lambda$, 
because we have the following implications.
\[
({\rm s2}^+)
\ \stackrel{}{\Longrightarrow} \ 
({\rm s2}^-)
\ \stackrel{\text{Prop}~\ref{prop:PreliminaryS}}{\Longrightarrow} \ 
({\rm s1})
\ \stackrel{\text{Prop}~\ref{prop:PreliminaryS}}{\Longrightarrow} \ 
({\rm s2}^+).
\]
So, from now on, 
we will not mention
conditions like (s2$^+$) or (s2$^-$)
(except for a condition in Theorem~\ref{th:application}).
}
\end{remark}

\smallskip

The completeness of {\HD} is described below,
which will be proved in 
Sections~\ref{sec:cut-free-completeness} and \ref{sec:HilbertD}.
\begin{proposition}[Completeness of {\HD}]
\label{prop:PreliminaryD}
Let $\lambda$ be a limit ordinal.
The following are equivalent.
\begin{itemize}
\item [{\rm (d1)}]
$\HD \vdash \varphi$.
\item [{\rm (d2)}]
$\varphi$ 
is true at the bottom of any constant $\lambda$-extension of any GL-model.
\item [{\rm (d3)}]
$\varphi$
is true at the bottom of any $\lambda$-extension of any GL-model.
\end{itemize}
\end{proposition}

\begin{remark}
\textup{
Beklemishev \cite{Beklemishev99} proved 
the equivalence between (d1), (d2) for $\lambda = \omega$, and 
another condition using `accumulating root'.
In Section~\ref{sec:cut-free-completeness},
we will show the equivalence 
between (d2) and (d3). 
In Section~\ref{sec:HilbertD},
we will show the equivalence 
between (d1) and others.
}
\end{remark}


\section{Sequent calculi}
\label{sec:seq-cal}

We introduce sequent calculi.
From now on, letters $\GA, \DE, \ldots$ denote
finite sets of formulas.
As usual, the expression `$\GA, \varphi, \psi, \DE$', for example, 
stands for the set $\GA \cup \{\varphi, \psi\} \cup \DE$,
and `$\Box\GA$' stands for the set 
$\{\Box\gamma \mid \gamma \in \GA\}$.
$\SF{\GA}$ denotes the set of all subformulas
of formulas in $\GA$.

We use three different arrows $\Rightarrow$, $\Rrightarrow$, and $\Darrow$
to define three kinds of sequents
$\SEQ{\GA}{\DE}$ (called {\em GL-sequent}),
$\SSEQ{\GA}{\DE}$ (called {\em S-sequent}),
and 
$\DSEQ{\GA}{\DE}$ (called {\em D-sequent}).

The following is the well-known sequent calculus for classical propositional logic; 
we call it {\LK}.
\begin{quote}
Initial sequents:
$\SEQ{\varphi}{\varphi}$
\mbox{ and }
$\SEQ{\BOT}{}$

Inference rules:
\[
\infer[\mbox{(weakening), where $\GA \subseteq \GA'$ and $\DE \subseteq \DE'$.}]{\SEQ{\GA'}{\DE'}}{\SEQ{\GA}{\DE}}
\]
\[
\infer[\mbox{($\IMP$left)}]{\SEQ{\varphi \IMP \psi, \GA}{\DE}}{
 \SEQ{\GA}{\DE, \varphi}
 &
 \SEQ{\psi, \GA}{\DE}
}
\quad
\infer[\mbox{($\IMP$right)}]{\SEQ{\GA}{\DE, \varphi \IMP \psi}}{
 \SEQ{\varphi, \GA}{\DE, \psi}
} 
\]
\[
\infer[\mbox{(cut)}]{\SEQ{\GA,\PI}{\DE,\SI}}{\SEQ{\GA}{\DE, \varphi} & \SEQ{\varphi, \PI}{\SI}}
\]
\end{quote}
Similarly the sequent calculus
{\LKS} (or {\LKD}, respectively)
consists of the above initial sequents and inference rules
where all `$\Rightarrow$' are replaced with `$\Rrightarrow$' (or `$\Darrow$', respectively).
Next we introduce five rules 
({\RuleGL}), ({\RuleS}), (\RuleLeftBox), ({\RuleDa}),
and ({\RuleDb}) as below.
\[
\infer[(\RuleGL)]
{\SEQ{\Box\GA}{\Box\varphi}}{
  \SEQ{\GA, \Box\GA, \Box\varphi}{\varphi}
}
\]
\[
\infer[(\RuleS)
\mbox{ \ (used in {\sS}, {\sDb})}]
{\SSEQ{\GA}{\DE}}{
  \SEQ{\GA}{\DE}
}
\qquad
\infer[(\RuleLeftBox)
\mbox{ \ (used in {\sS}, {\sDb})}]
{\SSEQ{\Box\varphi, \GA }{\DE}}{
  \SSEQ{\varphi, \GA}{\DE}
}  
\]
\[
\infer[(\RuleDa)
\mbox{ \ (used in {\sDa})}]
{\DSEQ{\Box\GA}{\Box\DE}}{
  \SEQ{\GA, \Box\GA}{\Box\DE}
}
\qquad
\infer[(\RuleDb)
\mbox{ \ (used in {\sDb})}]
{\DSEQ{\Box\GA}{\Box\DE}}{
  \SSEQ{\Box\GA}{\Box\DE}
}
\]
\begin{table}[h]
\caption{Four sequent calculi}
\label{table:seq-cal}
\begin{center}
\begin{tabular}{ccccccccccc}
\hline
&& {\LK} & {\LKS} & {\LKD} & 
& (\RuleGL) & (\RuleS) & (\RuleLeftBox) & (\RuleDa) & (\RuleDb)
\\
\hline
{\sGL} && 
$\checkmark$ & & &&
$\checkmark$ & & & &
\\
{\sS} &&
$\checkmark$ & $\checkmark$ & &&
$\checkmark$ & $\checkmark$ & $\checkmark$ & &
\\
{\sDa} &&
$\checkmark$ & & $\checkmark$ &&
$\checkmark$ & & & $\checkmark$ & 
\\
{\sDb} && 
$\checkmark$ & $\checkmark$ & $\checkmark$ &&
$\checkmark$ & $\checkmark$ & $\checkmark$  & & $\checkmark$ 
\\
\hline
\end{tabular} 
\end{center}
\end{table}
Using these, we define four sequent calculi
{\sGL}, {\sS}, {\sDa},  and {\sDb}
according to Table~\ref{table:seq-cal}.
{\sGL} ($ = {\LK} + (\RuleGL)$) is a well-known sequent calculus for {\bf GL},
and has been widely studied; see \cite{Gore} and its references.
{\sS} ($ = {\sGL} + (\RuleS) + (\RuleLeftBox) +  {\LKS}$) is a sequent calculus for {\bf S}
using two kinds of sequents (GL- and S-sequents),
and it (or similar calculi) has been studied in 
\cite{Beklemishev87, Kushida, KK}.
The two calculi {\sDa} and {\sDb} are new,
and they are obtained as below.
\[
{\sDa} = {\sGL} + (\RuleDa) + {\LKD}.
\]
\[
{\sDb} = {\sS} + (\RuleDb) + {\LKD}.
\]
The superscripts `2' and `3' denote the 
number of kinds of sequents;
that is,
{\sDa} uses two kinds (GL- and D-sequents)
and 
{\sDb} uses three kinds of sequents.

Proofs in each calculus are called 
{\sGL}/{\sS}/{\sDa}/{\sDb}-proofs.
In general, an {\sS}-proof has two-layered structure
(the upper layer consists of GL-sequents and
the lower layer consists of S-sequents).
Similarly, 
a {\sDa}-proof has two layers
(upper: GL-sequents; lower: D-sequents),
and 
a {\sDb}-proof has three layers
(top: GL-sequents; middle: S-sequents; bottom: D-sequents).
Note that 
only {\LKD}-rules can have D-sequents as assumptions.
Therefore the bottom layer of a {\sDa}/{\sDb}-proof 
consists of only {\LKD}-rules.

As usual, the term `cut-free' means `without using the cut rule'.
The following examples are
cut-free {\sDa}/{\sDb}-proofs,
where `w.' stands for `weakening', and 
($\OR$left), ($\OR$right), and ($\NOT$right) are suitable combinations of rules in {\bf LK} (since $\OR$ and $\NOT$ are abbreviations).

\begin{example}
\label{ex:sDa-proof}
\textup{
$\text{Cut-free \sDa} \vdash \DSEQ{}{\Box(\Box\varphi \OR \Box\psi) \IMP \Box\varphi \OR \Box\psi}$.
\[
\infer[\mbox{\rm ($\IMP$right)}]{\DSEQ{}{\Box(\Box\varphi \OR \Box\psi) \IMP \Box\varphi \OR \Box\psi}}{
 \infer[\mbox{\rm ($\OR$right)}]{\DSEQ{\Box(\Box\varphi \OR \Box\psi)}{\Box\varphi \OR \Box\psi}}{
  \infer[(\RuleDa)]{\DSEQ{\Box(\Box\varphi \OR \Box\psi)}{\Box\varphi, \Box\psi}}{
   \infer[\mbox{\rm (w.)}]{\SEQ{\Box\varphi \OR \Box\psi, \Box(\Box\varphi \OR \Box\psi)}{\Box\varphi, \Box\psi}}{
    \infer[\mbox{\rm ($\OR$left)}]{\SEQ{\Box\varphi \OR \Box\psi}{\Box\varphi, \Box\psi}}{ 
      \SEQ{\Box\varphi}{\Box\varphi}
       &
       \SEQ{\Box\psi}{\Box\psi}
    }
   }
  }
 }
}
\]
}
\end{example}
\begin{example}
\label{ex:sDb-proof}
\textup{
$\text{Cut-free \sDb} \vdash \DSEQ{}{\Box(\Box\varphi \OR \Box\psi) \IMP \Box\varphi \OR \Box\psi}$.
\[
\infer[\mbox{\rm ($\IMP$right)}]{\DSEQ{}{\Box(\Box\varphi \OR \Box\psi) \IMP \Box\varphi \OR \Box\psi}}{
 \infer[\mbox{\rm ($\OR$right)}]{\DSEQ{\Box(\Box\varphi \OR \Box\psi)}{\Box\varphi \OR \Box\psi}}{
  \infer[(\RuleDb)]{\DSEQ{\Box(\Box\varphi \OR \Box\psi)}{\Box\varphi, \Box\psi}}{
   \infer[(\RuleLeftBox)]{\SSEQ{\Box(\Box\varphi \OR \Box\psi)}{\Box\varphi, \Box\psi}}{
    \infer[\mbox{\rm ($\OR$left)}]{\SSEQ{\Box\varphi \OR \Box\psi}{\Box\varphi, \Box\psi}}{ 
      \SSEQ{\Box\varphi}{\Box\varphi}
       &
       \SSEQ{\Box\psi}{\Box\psi}
    }
   }
  }
 }
}
\]
}
\end{example}
\begin{example}
\label{ex:sDa-proof2}
\textup{
$\text{Cut-free \sDa} \vdash \DSEQ{}{\NOT\Box\BOT}$.
\[
\infer[\mbox{\rm ($\NOT$right)}]{\DSEQ{}{\NOT\Box\BOT}}{
 \infer[(\RuleDa)]{\DSEQ{\Box\BOT}{}}{
   \infer[\mbox{\rm (w.)}]{\SEQ{\BOT, \Box\BOT}{}}{
     \SEQ{\BOT}{}
   }
 }
}
\]
}
\end{example}

The correspondence between sequent calculi and Hilbert-style systems 
is the same as in classical logic, 
as below.

\begin{proposition}
\label{prop:Corespondence-GL}
The following are equivalent.
\begin{itemize}
\item
$\sGL \vdash \SEQ{\GA}{\DE}$. 
\item
$\HGL \vdash \bigwedge \GA \IMP \bigvee \DE$.
\end{itemize}
\end{proposition}

\begin{proposition}
\label{prop:Corespondence-S}
The following are equivalent.
\begin{itemize}
\item
$\sS \vdash \SSEQ{\GA}{\DE}$. 
\item
$\HS \vdash \bigwedge \GA \IMP \bigvee \DE$.
\end{itemize}
\end{proposition}

\begin{proposition}
\label{prop:Corespondence-D}
The following are equivalent.
\begin{itemize}
\item
$\sDa \vdash \DSEQ{\GA}{\DE}$. 
\item
$\sDb \vdash \DSEQ{\GA}{\DE}$. 
\item
$\HD \vdash \bigwedge \GA \IMP \bigvee \DE$.
\end{itemize}
\end{proposition}
Propositions~\ref{prop:Corespondence-GL} and \ref{prop:Corespondence-S}
are well-known and easy to prove.
Proposition~\ref{prop:Corespondence-D} will be shown 
in Section~\ref{sec:HilbertD}.

\begin{remark}
\textup{
{\sDb} is obtained from {\sS} by adding some rules,
while the logic {\bf D} is strictly weaker than {\bf S}.
}
\end{remark}

In the rest of this section, 
we mention the calculus for {\bf S} by Kushida\cite{Kushida},
which we call {\Kushida}.
Written in our notation, the calculus {\Kushida}
consists of {\LK}, {\LKS}, and the four rules below.
\[
\infer[(\KushidaGLL)]
{\SSEQ{\Box\varphi, \GA }{\DE}}{
  \SEQ{\varphi, \GA}{\DE}
}  
\quad
\infer[(\KushidaGLR)]
{\SEQ{\Box\GA,\Box\DE}{\Box\varphi}}{
  \SEQ{\Box\GA, \DE, \Box\varphi}{\varphi}
}
\]
\[
\infer[(\RuleLeftBox)]
{\SSEQ{\Box\varphi, \GA }{\DE}}{
  \SSEQ{\varphi, \GA}{\DE}
}  
\quad
\infer[(\KushidaSR)]
{\SSEQ{\Box\GA,\Box\DE}{\Box\varphi}}{
  \SEQ{\Box\GA, \DE, \Box\varphi}{\varphi}
}
\]
Consequently, {\Kushida} is obtained from {\sS} by
replacing the two rules $\{(\RuleGL), (\RuleS)\}$
with the three rules 
$\{(\KushidaGLL), (\KushidaGLR), (\KushidaSR)\}$.
Note that sequents in \cite{Kushida} consist of {\em multisets} of formulas, 
while our sequents consist of {\em sets} of formulas.
However, this difference is not critical
because of the existence of the contraction-rules in \cite{Kushida};
so $\GA, \DE,\ldots$ here  denote sets of formulas.

The following lemma can be easily proved by induction on 
{\Kushida}-proofs.
\begin{lemma}[The rule (\RuleS) is cut-free admissible in {\Kushida}]
\label{lm:Kushida}
If 
$\Kushida \vdash \SEQ{\GA}{\DE}$,
then
$\Kushida \vdash \SSEQ{\GA}{\DE}$.
If 
cut-free $\Kushida \vdash \SEQ{\GA}{\DE}$,
then
cut-free $\Kushida \vdash \SSEQ{\GA}{\DE}$.
\end{lemma}
Then we have the following theorem.
\begin{theorem}[{\Kushida} and {\sS} are equivalent]
\label{th:Kushida_equal_S}
Let ${\blacktriangleright} \in \{\Rightarrow, \Rrightarrow\}$.
$\Kushida \vdash \XSEQ{\GA}{\DE}$
if and only if 
$\sS \vdash \XSEQ{\GA}{\DE}$.
Cut-free $\Kushida \vdash \XSEQ{\GA}{\DE}$
if and only if 
cut-free $\sS \vdash \XSEQ{\GA}{\DE}$.
\end{theorem}
\begin{proof}
Each rule in {\Kushida} is cut-free derivable in {\sS};
for example, the rule (\KushidaSR) is shown below.
\[
\infer[(\RuleS)]{\SSEQ{\Box\GA,\Box\DE}{\Box\varphi}}{
  \infer[(\RuleGL)]{\SEQ{\Box\GA,\Box\DE}{\Box\varphi}}{
    \infer[\text{(w.)}]{\SEQ{\Box\GA,\Box\DE, \GA,\DE,\Box\varphi}{\varphi}}{
      \SEQ{\Box\GA, \DE, \Box\varphi}{\varphi}
    }
  }
}
\]
On the other hand, 
each rule in {\sS} is cut-free admissible in {\Kushida}
(the rule (\RuleS) is shown by the previous lemma, 
and the rule (\RuleGL) is trivial).
Hence, this Theorem~\ref{th:Kushida_equal_S}
is easily shown by induction.
\end{proof}


\section{Syntactic arguments}

In this section, we show some basic properties of our four calculi,
which can be shown by syntactic method (i.e., induction on proofs).

\medskip

While the logics {\bf S} and {\bf D} are proper extensions of {\bf GL},
the sequent calculi {\sS}, {\sDa}, and {\sDb} are conservative extensions of {\sGL} 
with respect to GL-sequents.
Moreover, {\sDb} is a conservative extension of {\sS} 
with respect to S-sequents.
These are shown by the following theorems. 

\begin{theorem}[Conservativity of provability of GL-sequents]
\label{th:GL-conservativity}
The following four conditions are equivalent:
${\sGL}\vdash \SEQ{\GA}{\DE}$;
${\sS}\vdash \SEQ{\GA}{\DE}$;
${\sDa}\vdash \SEQ{\GA}{\DE}$;
and 
${\sDb}\vdash \SEQ{\GA}{\DE}$.
\end{theorem}

\begin{theorem}[Conservativity of provability of S-sequents]
\label{th:S-conservativity}
The following two conditions are equivalent:
${\sS}\vdash \SSEQ{\GA}{\DE}$
and 
${\sDb}\vdash \SSEQ{\GA}{\DE}$.
\end{theorem}

Theorems~\ref{th:GL-conservativity} and \ref{th:S-conservativity}
are trivial because 
any {\sS}/{\sDa}/{\sDb}-proofs of GL-sequents
consist of only the rules of {\sGL},
and 
any {\sDb}-proofs of S-sequents
consist of only the rules of {\sS}.
These two theorems and their cut-free versions
(e.g., 
cut-free \sGL $\vdash \SEQ{\GA}{\DE}$
if and only if 
cut-free \sS $\vdash \SEQ{\GA}{\DE}$)
will be used implicitly from now on.

The fact `both the logics {\bf S} and {\bf D} are extensions of {\bf GL}'
is expressed by the following.

\begin{theorem}
\label{th:GL-extension}
If ${\sGL} \vdash \SEQ{\GA}{\DE}$, then 
we have ${\sS} \vdash \SSEQ{\GA}{\DE}$, 
${\sDa} \vdash \DSEQ{\GA}{\DE}$,
and 
${\sDb} \vdash \DSEQ{\GA}{\DE}$.
\end{theorem}
\begin{proof}
Suppose there is a {\sGL}-proof $P$ of $\SEQ{\GA}{\DE}$.
${\sS} \vdash \SSEQ{\GA}{\DE}$ is trivial because of the rule (\RuleS).
On the other hand, 
${\sDa}/{\sDb} \vdash \DSEQ{\GA}{\DE}$
is shown by induction on $P$.
If the last inference rule  of $P$ is (\RuleGL), 
then $\SEQ{\GA}{\DE}$ is of the form $\SEQ{\Box\PI}{\Box\varphi}$,
and we get {\sDa}- and {\sDb}-proofs as below
without using the induction hypothesis.
\[
\infer[(\RuleDa)]{\DSEQ{\Box\PI}{\Box\varphi}}{
 \infer[\mbox{(w.)}]{\SEQ{\PI, \Box\PI}{\Box\varphi}}{
   \infer*[P]{\SEQ{\Box\PI}{\Box\varphi}}{
  }
 }
}
\qquad
\infer[(\RuleDb)]{\DSEQ{\Box\PI}{\Box\varphi}}{
 \infer[(\RuleS)]{\SSEQ{\Box\PI}{\Box\varphi}}{
   \infer*[P]{\SEQ{\Box\PI}{\Box\varphi}}{
  }
 }
}
\]
In other words,
a {\sGL}-proof 
\[
\infer*[\mbox{{\LK}-rules}]{\SEQ{\GA}{\DE}}{
  \infer[(\RuleGL)]{\SEQ{\Box\PI}{\Box\varphi}}{
    \infer*{\SEQ{\PI, \Box\PI,\Box\varphi}{\varphi}}{
    }
  }
}
\]
is translated into the following {\sDa}- and {\sDb}-proofs.
\[
\infer*[\mbox{{\LKD}-rules}]{\DSEQ{\GA}{\DE}}{
 \infer[(\RuleDa)]{\DSEQ{\Box\PI}{\Box\varphi}}{
  \infer[\mbox{(w.)}]{\SEQ{\PI, \Box\PI}{\Box\varphi}}{
   \infer[(\RuleGL)]{\SEQ{\Box\PI}{\Box\varphi}}{
    \infer*{\SEQ{\PI, \Box\PI,\Box\varphi}{\varphi}}{
    }
   }
  }
 }
}
\qquad
\infer*[\mbox{{\LKD}-rules}]{\DSEQ{\GA}{\DE}}{
 \infer[(\RuleDb)]{\DSEQ{\Box\PI}{\Box\varphi}}{
  \infer[(\RuleS)]{\SSEQ{\Box\PI}{\Box\varphi}}{
   \infer[(\RuleGL)]{\SEQ{\Box\PI}{\Box\varphi}}{
    \infer*{\SEQ{\PI, \Box\PI,\Box\varphi}{\varphi}}{
    }
   }
  }
 }
}
\]
\end{proof}

The fact `logic {\bf S} is an extension of {\bf D}'
is expressed by the following theorem.

\begin{theorem}
\label{th:D-extension}
If ${\sDa} \vdash \DSEQ{\GA}{\DE}$, then 
${\sS} \vdash \SSEQ{\GA}{\DE}$.
If ${\sDb} \vdash \DSEQ{\GA}{\DE}$, then 
${\sS} \vdash \SSEQ{\GA}{\DE}$.
\end{theorem}
\begin{proof}
Roughly speaking, we translate 
the {\sDa}- and {\sDb}-proofs
\[
\infer*[\mbox{{\LKD}-rules}]{\DSEQ{\GA}{\DE}}{
  \infer[(\RuleDa)]{\DSEQ{\Box\PS}{\Box\PH}}{
    \infer*{\SEQ{\PS, \Box\PS}{\Box\PH}}{
    }
  }
}
\qquad
\infer*[\mbox{{\LKD}-rules}]{\DSEQ{\GA}{\DE}}{
  \infer[(\RuleDb)]{\DSEQ{\Box\PS}{\Box\PH}}{
    \infer*{\SSEQ{\Box\PS}{\Box\PH}}{
    }
  }
}
\]
into the following {\sS}-proofs respectively.
\[
\infer*[\mbox{{\LKS}-rules}]{\SSEQ{\GA}{\DE}}{
  \infer[(\RuleLeftBox)]{\SSEQ{\Box\PS}{\Box\PH}}{
   \infer[\text{(\RuleS)}]{\SSEQ{\PS, \Box\PS}{\Box\PH}}{  
    \infer*{\SEQ{\PS, \Box\PS}{\Box\PH}}{
    }
   }
  }
}
\qquad
\infer*[\mbox{{\LKS}-rules}]{\SSEQ{\GA}{\DE}}{
    \infer*{\SSEQ{\Box\PS}{\Box\PH}}{
    }
}
\]
To be precise, we show this Theorem~\ref{th:D-extension}
by induction on the {\sDa}- and {\sDb}-proofs
of $\DSEQ{\GA}{\DE}$.
\end{proof}

The following theorem is expected to be hold as a matter of course.

\begin{theorem}[Equivalence between {\sDa} and {\sDb}]
\label{th:sDa-sDb}
${\sDa}\vdash \DSEQ{\GA}{\DE}$
if and only if
${\sDb}\vdash \DSEQ{\GA}{\DE}$.
\end{theorem}
The only-if part is easily shown by induction on 
{\sDa}-proofs of $\DSEQ{\GA}{\DE}$;
we use the rules (\RuleS), (\RuleLeftBox), and (\RuleDb)
if the last inference rule is (\RuleDa).
On the other hand, 
the if-part is not trivial because
of the existence of S-sequents in {\sDb}-proofs.
We need some lemmas as below.
In the following, 
${\rm ref}(\SI)$ will denote 
the set $\{\Box\sigma \IMP \sigma \mid \sigma \in \SI\}$
(the name `ref' comes from `reflection').

\begin{lemma}
\label{lm:for-sDb-sDa-I}
If ${\sS} \vdash \SSEQ{\GA}{\DE}$, 
then there is a finite set $\SI$ of formulas such that 
${\sGL} \vdash \SEQ{{\rm ref}(\SI), \GA}{\DE}$.
\end{lemma}
\begin{proof}
By induction on the {\sS}-proof $P$ of $\SSEQ{\GA}{\DE}$.
If $P$ is 
\[
\infer[(\RuleLeftBox)]{\SSEQ{\Box\varphi, \PI}{\DE}}{
  \infer*[P']{\SSEQ{\varphi, \PI}{\DE}}{
  }
}
\]
then we have 
\[
\infer=[\mbox{(w.)($\IMP$left)}]{\SEQ{\Box\varphi \IMP \varphi, {\rm ref}(\SI'), \Box\varphi, \PI}{\DE}.}{
  \SEQ{\Box\varphi}{\Box\varphi}
  &
  \infer*[\mbox{i.h. for $P'$}]{\SEQ{{\rm ref}(\SI'), \varphi, \PI}{\DE}}{  
  }
}
\]
\end{proof}

\begin{lemma}
\label{lm:for-sDb-sDa-II}
If ${\sGL} \vdash \SEQ{{\rm ref}(\SI), \Box\PS}{\Box\PH}$,
then 
${\sDa} \vdash \DSEQ{\Box\PS}{\Box\PH}$.
\end{lemma}
\begin{proof}
It is easy to show that 
{\sGL} has the following property, which we call ($\IMP$left$^{-1}$):
\begin{em}
If ${\sGL} \vdash \SEQ{\varphi\IMP\psi, \GA}{\DE}$,
then 
${\sGL} \vdash \SEQ{\GA}{\DE, \varphi}$
and 
${\sGL} \vdash \SEQ{\psi, \GA}{\DE}$.
\end{em}
Now  
assume 
${\sGL} \vdash \SEQ{{\rm ref}(\SI), \Box\PS}{\Box\PH}$ 
where
$\SI = \{\sigma_1, \sigma_2, \ldots, \sigma_n\}$.
For any subset $\SI' \subseteq \SI$,
we have 
${\sDa} \vdash \DSEQ{\Box\SI', \Box\PS}{\Box\PH, \Box(\SI\setminus\SI')}$
as below.
\begin {align*}
&{\sGL} \vdash \SEQ{\Box\sigma_1\IMP\sigma_1, \ldots, \Box\sigma_n\IMP\sigma_n, \Box\PS}{\Box\PH}
\tag{Assumption of the Lemma}
\\
&{\sGL} \vdash \SEQ{\Sigma', \Box\PS}{\Box\PH, \Box(\SI\setminus\SI')}
\tag{By $n$-times application of ($\IMP$left$^{-1}$)}
\\
&{\sDa} \vdash \DSEQ{\Box\Sigma', \Box\PS}{\Box\PH, \Box(\SI\setminus\SI')}
\tag{By weakening and (\RuleDa)}
\end{align*}
Then, by using the cut rule
to the $2^n$ D-sequents,
we have ${\sDa} \vdash \DSEQ{\Box\PS}{\Box\PH}$.
For example, if $n=2$, the proof is below.
\[
\infer[\text{(cut)}]{\DSEQ{\Box\PS}{\Box\PH}}{
  \infer[\text{(cut)}]{\DSEQ{\Box\PS}{\Box\PH, \Box\sigma_1}}{
    \DSEQ{\Box\PS}{\Box\PH, \Box\sigma_1, \Box\sigma_2}
    &
    \DSEQ{\Box\sigma_2, \Box\PS}{\Box\PH, \Box\sigma_1}    
  }
  &
  \infer[\text{(cut)}]{\DSEQ{\Box\sigma_1, \Box\PS}{\Box\PH}}{
    \DSEQ{\Box\sigma_1, \Box\PS}{\Box\PH, \Box\sigma_2}
    &
    \DSEQ{\Box\sigma_1, \Box\sigma_2, \Box\PS}{\Box\PH}    
  }
}
\]
\end{proof}

\begin{proof}[Proof of if-part of Theorem~\ref{th:sDa-sDb}]
By induction on {\sDb}-proof $P$ of $\DSEQ{\GA}{\DE}$.
If $P$ is of the form
\[
\infer[\text{(\RuleDb)}]{\DSEQ{\Box\PS}{\Box\PH}}{
  \infer*[P']{\SSEQ{\Box\PS}{\Box\PH}}{}
}
\]
then we have 
${\sDa} \vdash \DSEQ{\Box\PS}{\Box\PH}$
by Lemmas~\ref{lm:for-sDb-sDa-I} and \ref{lm:for-sDb-sDa-II}.
\end{proof}

Two calculi {\sDa} and {\sDb} are equivalent as above;
however, they are not equivalent with respect to cut-eliminability.

\begin{theorem}[Cut-elimination for {\sDb}]
If $\sDb \vdash \DSEQ{\PS}{\PH}$,
then cut-free $\sDb \vdash \DSEQ{\PS}{\PH}$.
\end{theorem}

\begin{proof}
Combining 
cut-free versions of Theorems \ref{th:GL-conservativity} and \ref{th:S-conservativity}, 
Theorem~\ref{th:Kushida_equal_S}, and 
the syntactic cut-elimination for {\Kushida} 
(shown in \cite{Kushida}), 
we have cut-free admissibility of 
the cut rules for GL- and S-sequents in {\sDb};
that is, we have the following.
\begin{quote}
(cut$_{\Rightarrow}$) \ 
If both $\SEQ{\GA}{\DE, \varphi}$ 
and $\SEQ{\varphi, \PI}{\SI}$ are 
cut-free provable in $\sDb$,
then so is $\SEQ{\GA,\PI}{\DE,\SI}$.

(cut$_{\Rrightarrow}$) \ 
If both $\SSEQ{\GA}{\DE, \varphi}$ 
and $\SSEQ{\varphi, \PI}{\SI}$ are 
cut-free provable in $\sDb$,
then so is $\SSEQ{\GA,\PI}{\DE,\SI}$.
\end{quote}
Then we show this property  for D-sequents;
that is, we eliminate  
\[
\infer[\mbox{(cut$_{\Darrow}$)}]{\DSEQ{\GA,\PI}{\DE,\SI}}{\DSEQ{\GA}{\DE, \varphi} & \DSEQ{\varphi, \PI}{\SI}}
\]
As usual, this is shown by double induction on the length of the
cut-formula $\varphi$ and
sum of the length of the left and right upper proofs.
In the case of 
\[
\infer[\mbox{(cut$_{\Darrow}$)}]{\DSEQ{\Box\GA,\Box\PI}{\Box\DE,\Box\SI}}{
  \infer[({\RuleDb})]{\DSEQ{\Box\GA}{\Box\DE, \Box\varphi}}{
    \infer*[P]{\SSEQ{\Box\GA}{\Box\DE, \Box\varphi}}{}
  }
  & 
  \infer[({\RuleDb})]{\DSEQ{\Box\varphi, \Box\PI}{\Box\SI}}{
    \infer*[Q]{\SSEQ{\Box\varphi, \Box\PI}{\Box\SI}}{}
  }
}
\]
we have the following.
\[
\infer[({\RuleDb})]{\DSEQ{\Box\GA, \Box\PI}{\Box\DE, \Box\SI}}{
  \infer[\mbox{(cut$_\Rrightarrow$)}]{\SSEQ{\Box\GA, \Box\PI}{\Box\DE, \Box\SI}}{
    \infer*[P]{\SSEQ{\Box\GA}{\Box\DE, \Box\varphi}}{
    }
    &
    \infer*[Q]{\SSEQ{\Box\varphi, \Box\PI}{\Box\SI}}{
    }
 }
}
\]
The other cases are easily done by the standard method.
\end{proof}

A semantical proof of the above cut-elimination theorem 
will be given in the next section.

\begin{theorem}[Failure of cut-elimination for {\sDa}]
\label{th:failure-of-cuteli}
There is a D-sequent 
that is provable but not cut-free provable in {\sDa}.
\end{theorem}
\begin{proof}
We have a {\sDa}-proof: 
\[
\infer[\text{(cut)}]{\DSEQ{\Box\Box\Box p}{\Box p}}{
  \infer=[\text{(w.)(\RuleDa)}]{\DSEQ{\Box\Box\Box p}{\Box\Box p}}{
    \SEQ{\Box\Box p}{\Box\Box p}
  }
  &
  \infer=[\text{(w.)(\RuleDa)}]{\DSEQ{\Box\Box p}{\Box p}}{
    \SEQ{\Box p}{\Box p}
  }
}
\]
This last sequent is not cut-free provable 
because otherwise
the last inference should be one of the following rules.
\[
\infer[(\RuleDa)]{\DSEQ{\Box\Box\Box p}{\Box p}}{
  \SEQ{\Box\Box p, \Box\Box\Box p}{\Box p}
}
\quad
\infer[\text{(w.)}]{\DSEQ{\Box\Box\Box p}{\Box p}}{
  \DSEQ{}{\Box p}
}
\quad
\infer[\text{(w.)}]{\DSEQ{\Box\Box\Box p}{\Box p}}{
  \DSEQ{\Box\Box\Box p}{}
}
\quad
\infer[\text{(w.)}]{\DSEQ{\Box\Box\Box p}{\Box p}}{
  \DSEQ{}{}
}
\]
But none of these upper sequents are provable.
The fact
${\sGL} \not\vdash \SEQ{\Box\Box p, \Box\Box\Box p}{\Box p}$
can be shown using the soundness of {\sGL}
(i.e., combination of Propositions~\ref{prop:PreliminaryGL} and \ref{prop:Corespondence-GL})
and a GL-model $\Kmodel{\{x,y\}}{\R}{V}$
where $\R = \{(x,y)\}$ and $V(y, p) = \F$ (thus, 
$V(x, \Box p) = \F$ and $V(x, \Box\Box p) = V(x, \Box\Box\Box p) = \T$).
The other sequents 
$(\DSEQ{}{\Box p})$,
$(\DSEQ{\Box\Box\Box p}{})$,
and 
$(\DSEQ{}{})$
are easily shown to be cut-free unprovable.
\end{proof}

In the next section, 
we will show a weaker version of cut-elimination for {\sDa} --- 
we can eliminate all but the cuts that have subformula property ---
by a semantical method.


\section{Cut-free completeness}
\label{sec:cut-free-completeness}

In this section, we prove cut-free completeness of the sequent calculi.
As shown above, 
{\sDa} does not have the full cut-elimination property;
so, for {\sDa}, 
we show the completeness of `analytic {\sDa}'
which is defined below.

\begin{definition}
\textup{
{\em Analytic} {\sDa} is a restriction of {\sDa}
where cut is allowed only in the form
\[
\infer[\mbox{(cut$_{\Darrow}$) \ Proviso: $\Box\varphi \in \SF{\GA,\DE,\PI,\SI}$}.]{\DSEQ{\GA,\PI}{\DE,\SI}.}{\DSEQ{\GA}{\DE, \Box\varphi} & \DSEQ{\Box\varphi, \PI}{\SI}}
\]
}
\end{definition}

The following are key notions for proving 
the completeness.

\begin{definition}
\textup{
Let $\blacktriangleright \in \{\Rightarrow, \Rrightarrow, \Darrow\}$.
\begin{itemize}
\item
A GL/S/D-sequent $\XSEQ{\GA}{\DE}$  is 
{\em $\IMP$-saturated}
$\quad\DEF\quad$
\\
\hfill
$(\forall \varphi,\psi \in \FORMULA)
((\varphi \IMP \psi \in \GA \Longrightarrow \varphi \in \DE \text{ or } \psi \in \GA)
\text{ and } 
(\varphi \IMP \psi \in \DE \Longrightarrow \varphi \in \GA \text{ and } \psi \in \DE))$.
\item
A GL/S/D-sequent $\XSEQ{\GA}{\DE}$  is 
{\em $\Box$-left-saturated}
$\quad\DEF\quad$
$(\forall \varphi \in \FORMULA)
(\Box\varphi \in\GA \Longrightarrow \varphi \in \GA)$.
\item
A GL-sequent $\SEQ{\PS^+}{\PH^+}$ is a 
{\em \sGL-saturation of $\SEQ{\PS}{\PH}$}
$\quad\DEF\quad$
$\PS \subseteq \PS^+ \subseteq \SF{\PS,\PH}$; 
$\PH \subseteq \PH^+ \subseteq \SF{\PS,\PH}$; 
$\SEQ{\PS^+}{\PH^+}$ is $\IMP$-saturated;
and cut-free {\sGL} $\not\vdash \SEQ{\PS^+}{\PH^+}$.
\item
An S-sequent $\SSEQ{\PS^+}{\PH^+}$ is an
{\em \sS-saturation of $\SSEQ{\PS}{\PH}$}
$\quad\DEF\quad$
$\PS \subseteq \PS^+ \subseteq \SF{\PS,\PH}$; 
$\PH \subseteq \PH^+ \subseteq \SF{\PS,\PH}$; 
$\SSEQ{\PS^+}{\PH^+}$ is $\IMP$-saturated and $\Box$-left-saturated;
and 
cut-free {\sS} $\not\vdash \SSEQ{\PS^+}{\PH^+}$.
\item
A D-sequent $\DSEQ{\PS^+}{\PH^+}$ is a
{\em \sDa-saturation of $\DSEQ{\PS}{\PH}$}
$\quad\DEF\quad$
$\PS \subseteq \PS^+ \subseteq \SF{\PS,\PH}$; 
$\PH \subseteq \PH^+ \subseteq \SF{\PS,\PH}$; 
$\DSEQ{\PS^+}{\PH^+}$ is $\IMP$-saturated;
$\Box\varphi \in \PS^+\cup\PH^+$ for any $\Box\varphi \in \SF{\PS,\PH}$;
and analytic {\sDa} $\not\vdash \DSEQ{\PS^+}{\PH^+}$.
\item
A D-sequent $\DSEQ{\GA}{\DE}$ is a
{\em \sDb-saturation of $\DSEQ{\PS}{\PH}$}
$\quad\DEF\quad$
$\PS \subseteq \PS^+ \subseteq \SF{\PS,\PH}$; 
$\PH \subseteq \PH^+ \subseteq \SF{\PS,\PH}$; 
$\DSEQ{\PS^+}{\PH^+}$ is $\IMP$-saturated;
and 
cut-free {\sDb} $\not\vdash \DSEQ{\PS^+}{\PH^+}$.
\end{itemize}
}
\end{definition}

The following lemma is a standard tool for the semantical cut-eliminations.

\begin{lemma}
\label{lm:saturation}
\ \\
{\bf (1)}
If cut-free {\sGL} $\not\vdash \SEQ{\PS}{\PH}$,
then there is a \sGL-saturation of it.

\noi{\bf (2)}
If cut-free {\sS} $\not\vdash \SSEQ{\PS}{\PH}$,
then there is an \sS-saturation of it.

\noi{\bf (3)}
If analytic {\sDa} $\not\vdash \DSEQ{\PS}{\PH}$,
then there is a \sDa-saturation of it.

\noi{\bf (4)}
If cut-free {\sDb} $\not\vdash \DSEQ{\PS}{\PH}$,
then there is a \sDb-saturation of it.
\end{lemma}
\begin{proof}
A proof of (1) is well-known.
We add proper formulas to $\PS$ and $\PH$ step by step
(to be precise,
if $\varphi\IMP\psi$ is in the left-hand side,
we do either adding $\varphi$ to the right
or adding $\psi$ to the left;
if $\varphi\IMP\psi$ is in the right-hand side,
we add $\varphi$ to the left and $\psi$ to the right)
while preserving cut-free unprovability, 
and eventually we obtain an $\IMP$-saturated sequent.
The same proof can be done for (4).
For the proofs of (2) and (3), 
we combine the following step with the above procedure.
(For 2):
If $\Box\varphi$ is in the left-hand side,
then we add $\varphi$ there.
(For 3):
If $\Box\varphi \in \SF{\PS,\PH}$,
then 
we add $\Box\varphi$ to either the left-hand or right-hand side
while preserving analytic unprovability.
\end{proof}

\begin{theorem}[Soundness and cut-free completeness of {\sGL}]
\label{th:cut-free-completeness-GL}
For any GL-sequent $\SEQ{\PS}{\PH}$, the following three conditions are equivalent.
\begin{itemize}
\item[{\rm (1)}]
$\sGL \vdash \SEQ{\PS}{\PH}$.

\item[{\rm (2)}]
Cut-free $\sGL \vdash \SEQ{\PS}{\PH}$.

\item[{\rm (3)}]
$\bigwedge \PS \IMP \bigvee \PH$ is true 
at any world of any GL-model.
\end{itemize}
\end{theorem}

Theorem~\ref{th:cut-free-completeness-GL} is well-known;
see, e.g.,\cite{Avron} for the proof.

\medskip

Before showing the theorem for {\sS}, 
we state a proposition:

\begin{proposition}
\label{prop:trivial-GL}
If $\sGL \vdash \SEQ{\PS}{\PH}$, 
then 
$\bigwedge \PS \IMP \bigvee \PH$
is eventually always true in the tail of 
any $\lambda$-extension of any GL-model, 
for any limit ordinal $\lambda$.
\end{proposition}
This proposition is trivial
because a $\lambda$-extension of a GL-model is also a GL-model.
The reason for stating such a trivial proposition is as follows.
We will generalize the results of this section in Section~\ref{sec:generalization},
where the generalized proposition is not trivial but considered as a hypothesis.

\begin{theorem}[Soundness and cut-free completeness of {\sS}]
\label{th:cut-free-completeness-S}
Let $\lambda$ be a limit ordinal.
For any S-sequent $\SSEQ{\PS}{\PH}$, the following six conditions are equivalent.
\begin{itemize}
\item[{\rm (1)}]
$\sS \vdash \SSEQ{\PS}{\PH}$.

\item[{\rm (2)}]
Cut-free $\sS \vdash \SSEQ{\PS}{\PH}$.

\item[{\rm (3)}]
$\bigwedge \PS \IMP \bigvee \PH$ is true 
at the bottom of any strongly constant $\lambda$-extension of any GL-model.

\item[{\rm (4)}]
$\bigwedge \PS \IMP \bigvee \PH$ is eventually always true in the tail of 
any strongly constant $\lambda$-extension of any GL-model.

\item[{\rm (5)}]
$\bigwedge \PS \IMP \bigvee \PH$ is eventually always true in the tail of 
any constant $\lambda$-extension of any GL-model.

\item[{\rm (6)}]
$\bigwedge \PS \IMP \bigvee \PH$ is eventually always true in the tail of 
any $\lambda$-extension of any GL-model.
\end{itemize}
\end{theorem}

\begin{proof}
Note that the implications
`$(6)\Rightarrow (5) \Rightarrow (4)$'
and 
`$(2)\Rightarrow (1)$'
are trivial.

(Proof of $(1)\Rightarrow (6)$)
By induction on the {\sS}-proof of $\SSEQ{\PS}{\PH}$.
When the last inference rule is (\RuleS), 
we use Proposition~\ref{prop:trivial-GL}.
When the last inference rule is (\RuleLeftBox), 
we use the fact that
{\em $\Box\varphi \IMP \varphi$
is eventually always true in the tail 
$\{t_\alpha \mid 0 < \alpha < \lambda\}$ of 
any $\lambda$-extension of any Kripke model.}
(Proof of this fact:
If $\varphi$ is true at all the worlds in the tail, 
then so is $\Box\varphi \IMP \varphi$.
If $\varphi$ is false at $t_\alpha$ for some $\alpha$, then 
$\Box\varphi \IMP \varphi$ is true at $t_\beta$ for any $\beta$
such that $\alpha < \beta < \lambda$.)

(Proof of $(4)\Rightarrow (3)$)
Let $M = \Kmodel{W}{\R}{V}$ be a strongly constant 
$\lambda$-extension of a Kripke-model,
$t_\lambda$ be the bottom, 
and $\{t_\alpha \mid 0 < \alpha < \lambda\}$ be the tail of $M$.
For any formula $\varphi$, we consider four conditions below.
($\lambda$T):
$V(t_\lambda, \varphi) = \T$.
($\lambda$F):
$V(t_\lambda, \varphi) = \F$.
(ET):
$(\exists \alpha <\lambda)(\alpha < \forall \beta < \lambda)(V(t_\beta, \varphi) = \T)$.
(EF):
$(\exists \alpha <\lambda)(\alpha < \forall \beta < \lambda)(V(t_\beta, \varphi) = \F)$.
By induction on $\varphi$, we can show two implications
`$\text{($\lambda$T)} \Rightarrow \text{(ET)}$'
and
`$\text{($\lambda$F)} \Rightarrow \text{(EF)}$'.
Therefore we have
`$\text{(ET)} \Rightarrow \text{($\lambda$T)}$'
because 
the conditions (ET) and (EF) are exclusive, 
and the condition ($\lambda$T) is the negation of ($\lambda$F).

(Proof of $(3)\Rightarrow (2)$)
We show the contraposition.
Suppose cut-free $\sS \not\vdash \SSEQ{\PS}{\PH}$.
We will construct a strongly constant $\lambda$-extension of a GL-model 
such that $\bigwedge \PS \IMP \bigvee \PH$ is false at the bottom.
First we apply Lemma~\ref{lm:saturation}(2) to $\SSEQ{\PS}{\PH}$, 
and we get an \sS-saturation $\SSEQ{\PS^+}{\PH^+}$.
The GL-sequent $\SEQ{\PS^+}{\PH^+}$ is not cut-free provable because of 
the (\RuleS)-rule,
and this sequent is not cut-free provable also in {\sGL}.
Then we get a GL-model 
$\Kmodel{W}{\R}{V}$ such that 
$\bigwedge \PS^+ \IMP \bigvee \PH^+$
is false at a world $t_0$, 
by the cut-free completeness of {\sGL} 
(Th.\ref{th:cut-free-completeness-GL}).
Thus, we have the following:
(A) If $p \in \PS^+$, then $V(t_0,p)=\T$.
(B) If $p \in \PH^+$, then $V(t_0,p)=\F$.
(C) If $\Box\psi \in \PS^+$, then $V(x,\psi)=\T$ for any $x \in W$
such that $t_0 \R x$.
(D) If $\Box\psi \in \PH^+$, then $V(x,\psi)=\F$ for some $x \in W$
such that $t_0 \R x$.
Then we define a strongly constant $\lambda$-extension 
$\Kmodel{W^+}{\R^+}{V^+}$ of $\Kmodel{W}{\R}{V}$ 
where $\{t_\alpha \mid 0 < \alpha < \lambda\}$ 
is the tail, $t_\lambda$ is the bottom,
and 
$V^+(t_\alpha, p) = V(t_0, p)$ 
for any $\alpha \leq \lambda$ and $p$.
(In Figure~\ref{Fig:Completeness} (left),
the essence of $\Kmodel{W^+}{\R^+}{V^+}$
is illustrated:
each world is a sequent,
and the goal is to show that 
each formula in the left (or right) hand side is 
true (or false, resp.) there.)
Now, for any formula $\varphi$ and any ordinal number $\alpha \leq \lambda$, 
we prove the following two properties by induction on $\varphi$.
(L) {\em If $\varphi \in \PS^+$, then $V^+(t_\alpha, \varphi) = \T$.}
(R) {\em If $\varphi \in \PH^+$, then $V^+(t_\alpha, \varphi) = \F$.}
In the case of $\varphi = p$, 
we use the facts (A) and (B).
In the case of $\varphi = \Box\psi$, we use 
$\Box$-left-saturatedness, the induction hypothesis, and 
the fact (C) to show the property (L);
and  we use the fact (D) to show the property (R).
Thus we have 
$V^+(t_\lambda, \bigwedge \PS \IMP \bigvee \PH) = \F$.
\end{proof}

\begin{figure}[h]

\begin{minipage}{7cm}
\begin{tikzpicture} [x=1mm,y=1mm]

\draw(0,10) ellipse (15 and 7.5);

\node at (0,6) {$\SEQ{\PS^+}{\PH^+}$};
\draw [->, thick] (-1,7)--(-4,9);
\draw [->, thick] (0,7)--(0,10);
\draw [->, thick] (1,7)--(4,9);

\draw [->, thick] (0,1.5)--(0,4);
\node at (0,0) {$\SSEQ{\PS^+}{\PH^+}$};
\draw [->, thick] (0,-4.5)--(0,-2);
\node at (0,-6) {$\SSEQ{\PS^+}{\PH^+}$};
\draw [->, thick] (0,-10.5)--(0,-8);

\node at (0,-11) {$\cdot$};
\node at (0,-12) {$\cdot$};
\node at (0,-13) {$\cdot$};
\node at (0,-14) {$\cdot$};
\node at (0,-15) {$\cdot$};

\node at (0,-19) {$\SSEQ{\PS^+}{\PH^+}$};

\node [right] at (5,12) {\colorbox{white}{$\Kmodel{W}{\R}{V}$}};

\node at (-40,0) {};
\end{tikzpicture} 
\end{minipage}
\begin{minipage}{7cm}
\begin{tikzpicture} [x=1mm,y=1mm]

\draw(0,10) ellipse (15 and 7.5);

\node at (0,6) {$\SEQ{\PI}{\SI}$};
\draw [->, thick] (-1,7)--(-4,9);
\draw [->, thick] (0,7)--(0,10);
\draw [->, thick] (1,7)--(4,9);

\draw [->, thick] (0,1.5)--(0,4);
\node at (0,0) {$\SEQ{\PI}{\SI}$};
\draw [->, thick] (0,-4.5)--(0,-2);
\node at (0,-6) {$\SEQ{\PI}{\SI}$};
\draw [->, thick] (0,-10.5)--(0,-8);

\node at (-7,-9) {\small \rotatebox{45}{$\subseteq$}};
\node at (-8,-13) {\small $\PS^\star, \Box\PS^\star$};

\node at (7,-9) {\small \rotatebox{-45}{$\supseteq$}};
\node at (10.5,-12.5) {\small $\Box\PH^\star$};

\node at (0,-11) {$\cdot$};
\node at (0,-12) {$\cdot$};
\node at (0,-13) {$\cdot$};
\node at (0,-14) {$\cdot$};
\node at (0,-15) {$\cdot$};

\node at (0,-19) {$\DSEQ{\PS^+}{\PH^+}$};

\node [right] at (5,12) {\colorbox{white}{$\Kmodel{W}{\R}{V}$}};

\node at (-40,0) {};
\end{tikzpicture} 
\end{minipage}
\caption{
$\Kmodel{W^+}{\R^+}{V^+}$ in Theorem~\ref{th:cut-free-completeness-S} 
(left)  and in Theorem~\ref{th:cut-free-completeness-Da} (right).
}
\label{Fig:Completeness}
\end{figure}

\begin{theorem}[Soundness and analytic completeness of {\sDa}]
\label{th:cut-free-completeness-Da}
Let $\lambda$ be a limit ordinal.
For any D-sequent $\DSEQ{\PS}{\PH}$, the following four conditions are equivalent.
\begin{itemize}
\item[{\rm (1)}]
$\sDa\vdash \DSEQ{\PS}{\PH}$.

\item[{\rm (2)}]
Analytic $\sDa \vdash \DSEQ{\PS}{\PH}$.

\item[{\rm (3)}]
$\bigwedge \PS \IMP \bigvee \PH$ is true 
at the bottom of any constant $\lambda$-extension of any GL-model.

\item[{\rm (4)}]
$\bigwedge \PS \IMP \bigvee \PH$ is true 
at the bottom of any $\lambda$-extension of any GL-model.
\end{itemize}
\end{theorem}

\begin{proof}
The implications `$(2)\Rightarrow (1)$' 
and `$(4) \Rightarrow (3)$' are trivial.

(Proof of $(1) \Rightarrow (4)$)
Let $M^+ = \Kmodel{W^+}{\R^+}{V^+}$ be a $\lambda$-extension of a GL-model.
We prove, 
by induction on the {\sDa}-proof of $\DSEQ{\PS}{\PH}$,
that 
$\bigwedge \PS \IMP \bigvee \PH$ is true at the bottom $t_\lambda$ of $M^+$.
Here we show the case that 
$(\DSEQ{\PS}{\PH}) = (\DSEQ{\Box\PS'}{\Box\PH'})$ 
and the last inference rule is (\RuleDa);
that is, we have the hypothesis
$(\dag) \ {\sGL} \vdash \SEQ{\PS', \Box\PS'}{\Box\PH'}$.
We assume 
$\bigwedge\Box\PS'$ is true at $t_\lambda$;
then, the goal is to show 
$\bigvee\Box\PH'$ is true at $t_\lambda$.
By the assumption and the definition of $\R^+$, 
we have both
$\bigwedge\PS'$
and 
$\bigwedge\Box\PS'$
are true at any world in the tail.
Then the hypothesis $(\dag)$
and Proposition~\ref{prop:trivial-GL}
imply that 
$\bigvee\Box\PH'$ is eventually always true in the 
tail $\{t_\alpha \mid 0<\alpha<\lambda\}$.
Therefore, there must be a formula $\Box\varphi \in \Box\PH'$
such that the condition 
(\ddag) $(\forall \alpha < \lambda)(\alpha < \exists \beta < \lambda) 
(V^+(t_\beta,\Box\varphi)=\T)$
holds
because otherwise
every $\Box\varphi \in \Box\PH'$ satisfies 
$(\exists \alpha < \lambda)(\alpha < \forall \beta < \lambda) 
(V^+(t_\beta,\Box\varphi)=\F)$,
and this contradicts the fact that 
$\bigvee\Box\PH'$ is eventually always true.
This formula $\Box\varphi$ is also true at the bottom $t_\lambda$
because the condition (\ddag) implies 
$(\forall \alpha < \lambda)[(V^+(t_\alpha,\varphi)=\T)$
and
$(\forall x \RI^+ t_\alpha)(V^+(x,\varphi)=\T)]$.
Thus, we have 
$V^+(t_\lambda, \bigvee\Box\PH')=\T$.

(Proof of $(3) \Rightarrow (2)$)
We show the contraposition.
Suppose analytic $\sDa \not\vdash \DSEQ{\PS}{\PH}$.
We will construct a constant $\lambda$-extension of a GL-model 
such that $\bigwedge \PS \IMP \bigvee \PH$ is false at the bottom.
First we apply Lemma~\ref{lm:saturation}(3) to $\DSEQ{\PS}{\PH}$, 
and we get a \sDa-saturation $\DSEQ{\PS^+}{\PH^+}$.
Let $\PS^\star = \{\psi \mid \Box\psi \in \PS^+\}$
and
$\PH^\star = \{\varphi \mid \Box\varphi \in \PH^+\}$.
The GL-sequent 
$\SEQ{\PS^\star, \Box\PS^\star}{\Box\PH^\star}$
is not cut-free provable in {\sGL}
because otherwise analytic $\sDa \vdash \DSEQ{\PS^+}{\PH^+}$
by the rules (\RuleDa) and weakening.
Then applying Lemma~\ref{lm:saturation}(1), we get 
a \sGL-saturation $\SEQ{\PI}{\SI}$ of $\SEQ{\PS^\star, \Box\PS^\star}{\Box\PH^\star}$.
We show that $\SEQ{\PI}{\SI}$ is $\Box$-left-saturated:
\begin{quote}
Suppose $\Box\pi \in \PI$.
Then $\Box\pi \in \SF{\PS,\PH}$
(because $\SEQ{\PI}{\SI}$ is a 
\sGL-saturation of $\SEQ{\PS^\star, \Box\PS^\star}{\Box\PH^\star}$
which consists of the elements of $\SF{\PS,\PH}$),
and we have either 
$\Box\pi \in \PS^+$
or 
$\Box\pi \in \PH^+$
(because $\DSEQ{\PS^+}{\PH^+}$ is a {\sDa}-saturation of $\DSEQ{\PS}{\PH}$).
But the latter condition $\Box\pi \in \PH^+$ does not hold
because, if $\Box\pi \in \PH^+$, then $\Box\pi \in \Box\PH^\star$
and $\SEQ{\PI}{\SI}$ is cut-free provable 
from the initial sequent $\SEQ{\Box\pi}{\Box\pi}$ by 
the weakening rule in {\sGL}.
Therefore the former condition $\Box\pi \in \PS^+$ holds,
and we have $\pi \in \PS^\star \subseteq \PI$. 
\end{quote}
The remaining part is similar to the proof of 
`$(3) \Rightarrow (2)$' of Theorem~\ref{th:cut-free-completeness-S}.
We get a GL-model 
$\Kmodel{W}{\R}{V}$ such that 
$\bigwedge \PI \IMP \bigvee \SI$
is false at a world $t_0$, 
by the cut-free completeness of {\sGL}.
Then we define a constant $\lambda$-extension of 
$\Kmodel{W^+}{\R^+}{V^+}$  of $\Kmodel{W}{\R}{V}$ 
where $\{t_\alpha \mid 0 < \alpha < \lambda\}$ 
is the tail, $t_\lambda$ is the bottom,
$V^+(t_\alpha, p) = V(t_0, p)$ for any $\alpha < \lambda$,
and 
$V^+(t_\lambda, p) = \T  \Longleftrightarrow p \in \PS^+$.
(In Figure~\ref{Fig:Completeness} (right),
the essence of $\Kmodel{W^+}{\R^+}{V^+}$
is illustrated.)
We consider four properties below, where $\alpha < \lambda$.
(L) {\em If $\varphi \in \PI$, then $V^+(t_\alpha, \varphi) = \T$.}
(R) {\em If $\varphi \in \SI$, then $V^+(t_\alpha, \varphi) = \F$.}
(L$\lambda$) {\em If $\varphi \in \PS^+$, then $V^+(t_\lambda, \varphi) = \T$.}
(R$\lambda$) {\em If $\varphi \in \PH^+$, then $V^+(t_\lambda, \varphi) = \F$.}
The properties (L) and (R)
are proved in the same way as that in Theorem~\ref{th:cut-free-completeness-S} for any ordinal number $\alpha < \lambda$.
Note that 
$\SEQ{\PI}{\SI}$ is $\Box$-left-saturated 
as shown above
(while 
$\DSEQ{\PS^+}{\PH^+}$ may not).
Finally, we prove the properties (L$\lambda$) and (R$\lambda$)
by induction on $\varphi$.
In the case of $\varphi = \Box\psi \in \PS^+$, for example,
the proof is done as follows. 
\begin{equation}
\Box\psi \in \PS^+
\ \stackrel{\text{def.~of~}\PI}{\Longrightarrow} \ 
\psi \in \PI
\ \stackrel{\text{(L)}}{\Longrightarrow} \ 
V^+(t_\alpha, \psi) = \T \ \text{ for any } \alpha < \lambda.
\tag{\dag}
\end{equation}
\begin{equation}
  \begin{split} 
\Box\psi \in \PS^+
\ \stackrel{\text{def.~of~}\PI}{\Longrightarrow} \ 
\Box\psi \in \PI
\ \stackrel{\text{(L)}}{\Longrightarrow} \ 
V^+(t_0, \Box\psi) = \T
\ \Longrightarrow \ 
\\
V^+(x,\psi)=\T \  \text{ for any } x \in W \text{ such that } t_0 \R x.
\end{split}
\tag{\ddag}
\end{equation}
\begin{equation*}
\Box\psi \in \PS^+
\ \stackrel{(\dag)(\ddag)}{\Longrightarrow} \ 
V^+(t_\lambda,\Box\psi)=\T.
\end{equation*}
\end{proof}

\begin{theorem}[Soundness and cut-free completeness of {\sDb}]
\label{th:cut-free-completeness-Db}
Let $\lambda$ be a limit ordinal.
For any D-sequent $\DSEQ{\PS}{\PH}$, the following four conditions are equivalent.
\begin{itemize}
\item[{\rm (1)}]
$\sDb\vdash \DSEQ{\PS}{\PH}$.

\item[{\rm (2)}]
Cut-free $\sDb \vdash \DSEQ{\PS}{\PH}$.

\item[{\rm (3)}]
$\bigwedge \PS \IMP \bigvee \PH$ is true 
at the bottom of any constant $\lambda$-extension of any GL-model.

\item[{\rm (4)}]
$\bigwedge \PS \IMP \bigvee \PH$ is true 
at the bottom of any $\lambda$-extension of any GL-model.
\end{itemize}
\end{theorem}

\begin{proof}
Proof can be done in a similar way to Theorem~\ref{th:cut-free-completeness-Da}.
Points of differences between two proofs are below.
In the proof of `$(1)\Rightarrow(4)$',
we use `$(1)\Rightarrow(6)$' of Theorem~\ref{th:cut-free-completeness-S}
instead of Proposition~\ref{prop:trivial-GL}.
In the proof of `$(3)\Rightarrow(2)$',
we use a \sDb-saturation $\DSEQ{\PS^+}{\PH^+}$ of $\DSEQ{\PS}{\PH}$ 
(by Lemma~\ref{lm:saturation}(4))
to define the valuation $V^+$ at the bottom;
and 
we use an \sS-saturation $\SSEQ{\PI}{\SI}$ of
$\SSEQ{\Box\PS^\star}{\Box\PH^\star}$ 
(and the GL-sequent $\SEQ{\PI}{\SI}$ which is not cut-free provable in {\sGL}) 
to define the valuation $V^+$ at the tail.
\end{proof}


\section{Hilbert-style systems for {\bf D}}
\label{sec:HilbertD}

Corresponding to {\sDa} and {\sDb}, 
we introduce new Hilbert-style proof systems,
called {\HDa} and {\HDb}:
\begin{quote}
Axioms of {\HDa}: 
Tautologies, 
and formulas of the form 
$\bigwedge \Box\GA \IMP \bigvee \Box\DE$
such that 
$\bigwedge(\GA, \Box\GA) \IMP \bigvee \Box\DE$
is a theorem of {\HGL}.
\\
Rule of {\HDa}: Modus ponens.

Axioms of {\HDb}: 
Tautologies, 
and formulas of the form 
$\bigwedge \Box\GA \IMP \bigvee \Box\DE$
which is a  theorem of {\HS}.
\\
Rule of {\HDb}: Modus ponens.
\end{quote}
Similarly to 
Propositions~\ref{prop:Corespondence-GL} and \ref{prop:Corespondence-S},
we have the following.

\begin{proposition}
\label{prop:Corespondence-Da}
$\sDa \vdash \DSEQ{\GA}{\DE}$
if and only if 
$\HDa \vdash \bigwedge \GA \IMP \bigvee \DE$.
\end{proposition}

\begin{proposition}
\label{prop:Corespondence-Db}
$\sDb \vdash \DSEQ{\GA}{\DE}$
if and only if
$\HDb \vdash \bigwedge \GA \IMP \bigvee \DE$.
\end{proposition}

These are proved by induction.
To be exact, for the if-parts,
we show 
`$\sDi \vdash (\DSEQ{\!}{\varphi})$
if 
$\HDi \vdash \varphi$'
by induction on the {\HDi}-proof of $\varphi$.
Note that 
Propositions~\ref{prop:Corespondence-GL} and \ref{prop:Corespondence-S}
and 
Theorems~\ref{th:GL-conservativity}, \ref{th:S-conservativity}, and \ref{th:GL-extension}
are used in the proofs.

Now the logic {\bf D} has three Hilbert-style systems
{\HD}, {\HDa} and {\HDb};
and the equivalence of these is shown 
by combining 
Proposition~\ref{prop:PreliminaryD} 
(for $\lambda = \omega$),
Theorems~\ref{th:cut-free-completeness-Da}
and \ref{th:cut-free-completeness-Db} 
(for $\lambda = \omega$),
and
Propositions~\ref{prop:Corespondence-Da}
and 
\ref{prop:Corespondence-Db}.
But it is a natural question whether we can show this equivalence syntactically (since the above combination uses soundness and completeness).
Syntactic proof of equivalence between
{\HDa} and {\HDb}
can be obtained by Theorem~\ref{th:sDa-sDb}
and 
Propositions~\ref{prop:Corespondence-Da}
and 
\ref{prop:Corespondence-Db}.
In the following, 
we show a syntactic proof of equivalence between
{\HDa} and {\HD};
consequently the equivalence between the three systems is shown syntactically.

\begin{theorem}
\label{th:Correspondence-D-Da}
$\HDa \vdash \varphi$
if and only if 
$\HD \vdash \varphi$.
\end{theorem}
\begin{proof}
(Proof of only-if-part)
We show the following by induction on \HDa-proof of $\varphi$:
{\em If $\HDa \vdash \varphi$, then $\HD \vdash \varphi$.}
If $\varphi$ is the axiom
$\bigwedge \Box\GA \IMP \bigvee \Box\DE$,
an outline of the proof is as follows.
\begin{align}
\text{(1)}\quad &
{\HGL}\vdash \bigwedge(\GA,\Box\GA) \IMP \bigvee \Box\DE
\tag{By definition of the axiom}
\\
\text{(2)}\quad &
{\HGL}\vdash \bigwedge(\Box\GA,\Box\Box\GA) \IMP\Box\bigvee \Box\DE
\tag{By (1) and inference in {\HGL}}
\\
\text{(3)}\quad &
{\HGL}\vdash \bigwedge\Box\GA \IMP\Box\bigvee \Box\DE
\tag{By (2) and the fact ${\HGL}\vdash \Box\gamma \IMP \Box\Box\gamma$}
\\
\text{(4)}\quad &
{\HD}\vdash \Box\bigvee \Box\DE \IMP \bigvee \Box\DE
\tag{Lemma 2.7.1 in \cite{Beklemishev99}}
\\
\text{(5)}\quad &
{\HD}\vdash \bigwedge \Box\GA \IMP \bigvee \Box\DE
\tag{By (3), (4), and definition of the axioms of {\HD}.}
\end{align}

(Proof of if-part)
First we show the following claim by induction on the \HGL-proof of $\varphi$:
{\em If $\HGL \vdash \varphi$, then $\HDa \vdash \varphi$.}
A point is that the axioms 
$\Box(\varphi \IMP \psi) \IMP (\Box\varphi \IMP \Box\psi)$
and 
$\Box(\Box\varphi \IMP \varphi) \IMP \Box\varphi$
and the conclusion of $(\Box)$-rule are 
respectively equivalent (over classical propositional logic) to 
formulas of the form 
$\bigwedge \Box\GA \IMP \bigvee \Box\DE$
such that 
$\bigwedge (\GA, \Box\GA) \IMP \bigvee \Box\DE$
is a theorem of {\HGL}.
For example, 
$\Box(\varphi \IMP \psi) \IMP (\Box\varphi \IMP \Box\psi)$
is equivalent to 
$\bigwedge \Box(\varphi \IMP \psi, \varphi) \IMP \Box\psi$,
where 
$\bigwedge(\varphi\IMP\psi, \varphi, \Box(\varphi \IMP \psi, \varphi)) \IMP \Box\psi$
is a theorem of {\HGL}.
Then we show the following claim by induction on the \HD-proof of $\varphi$:
{\em If $\HD \vdash \varphi$, then $\HDa \vdash \varphi$.}
For the axioms
$\NOT\Box\BOT$
and 
$\Box(\Box\varphi \OR \Box\psi) \IMP \Box\varphi \OR \Box\psi$, 
proofs are done corresponding to the {\sDa}-proofs in 
Examples~\ref{ex:sDa-proof2} and \ref{ex:sDa-proof}.
\end{proof}


\section{Generalization}
\label{sec:generalization}

In this section, 
we generalize the completeness of {\HS}, {\HDa},  and {\HDb}.

In the following, $L$ denotes an arbitrary set of formulas.
We define five sets of formulas  depending on $L$.

\begin{definition}[$\MP{\cdot}, \SL{\cdot}, \DL{\cdot}, \DaL{\cdot}, \DbL{\cdot}$]
\textup{
\begin{quote}
$\MP{L}$ is the smallest set of formulas that 
contains $L$ and is closed under modus ponens;
\\
$\SL{L} = \MP{L \cup \{\Box\varphi \IMP \varphi \mid \varphi \in \FORMULA\}}$;
\\
$\DL{L} = \MP{L \cup \{\NOT\Box\BOT\} \cup 
\{\Box(\Box\varphi \OR \Box\psi) \IMP \Box\varphi \OR \Box\psi \mid \varphi, \psi\in \FORMULA\}}$;
\\
$\DaL{L} = \MP{
{\bf Taut}  \cup 
\{\bigwedge \Box\GA \IMP \bigvee \Box\DE \mid
\bigwedge(\GA, \Box\GA) \IMP \bigvee \Box\DE \in L\}}$; and 
\\
$\DbL{L} = \MP{
{\bf Taut}  \cup 
\{\bigwedge \Box\GA \IMP \bigvee \Box\DE \mid
\bigwedge \Box\GA \IMP \bigvee \Box\DE \in \SL{L}\}}$;
\end{quote}
where ${\bf Taut}$ is the set of tautologies.
}
\end{definition}
For example, 
$\SL{{\bf GL}}$, $\DL{{\bf GL}}$, $\DaL{{\bf GL}}$,  and $\DbL{{\bf GL}}$
are the sets of theorems of 
{\HS}, {\HD}, {\HDa}, and {\HDb}, respectively,
if {\bf GL} is the set of theorems of {\HGL}.

\begin{definition}
\textup{
Let ${\cal F}$ be a class of Kripke frames
and $L$ be a set of formulas.
We say that a Kripke model is an {\em ${\cal F}$-model}
if it is based on a frame in ${\cal F}$.
We say that $L$ is {\em characterized by} ${\cal F}$
if the following two conditions are equivalent for any formula $\varphi$.
\begin{itemize}
\item
$\varphi \in L$.
\item
$\varphi$ is true at any world of any ${\cal F}$-model.
\end{itemize}
Moreover,
we say that $L$ is 
{\em $\lambda$-tail-sound for ${\cal F}$}
if any formula in $L$ 
is eventually always true in the tail of
any $\lambda$-extension of any ${\cal F}$-model, 
where $\lambda$ is a limit ordinal.
}
\end{definition}
For example, 
{\bf GL}
is characterized by
and 
$\lambda$-tail-sound for 
the class of GL-frames.

The following are the generalization of 
the completeness of {\HS}, {\HDa},  and {\HDb}.

\begin{theorem}
\label{th:generalization-S}
Let $\lambda$ be a limit ordinal.
If $L$ is characterized by ${\cal F}$
and $\lambda$-tail-sound for ${\cal F}$,
then  the following five conditions are equivalent.
\begin{itemize}
\item[{\rm (1)}]
$\varphi \in \SL{L}$.

\item[{\rm (2)}]
$\varphi$ is true 
at the bottom of any strongly constant $\lambda$-extension of any ${\cal F}$-model.

\item[{\rm (3)}]
$\varphi$ is eventually always true in the tail of 
any strongly constant $\lambda$-extension of any ${\cal F}$-model.

\item[{\rm (4)}]
$\varphi$ is eventually always true in the tail of 
any constant $\lambda$-extension of any ${\cal F}$-model.

\item[{\rm (5)}]
$\varphi$ is eventually always true in the tail of 
any $\lambda$-extension of any ${\cal F}$-model.
\end{itemize}
\end{theorem}

\begin{theorem}
\label{th:generalization-D}
Let $\lambda$ be a limit ordinal.
If $L$ is characterized by ${\cal F}$
and $\lambda$-tail-sound for ${\cal F}$,
then  the following four conditions are equivalent.
\begin{itemize}
\item[{\rm (1)}]
$\varphi \in \DaL{L}$.

\item[{\rm (2)}]
$\varphi \in \DbL{L}$.

\item[{\rm (3)}]
$\varphi$ is true 
at the bottom of any constant $\lambda$-extension of any ${\cal F}$-model.

\item[{\rm (4)}]
$\varphi$ is true 
at the bottom of any $\lambda$-extension of any ${\cal F}$-model.
\end{itemize}
\end{theorem}
\begin{proof}[Proof of Theorems~\ref{th:generalization-S} and \ref{th:generalization-D}]
Suppose that $L$ is characterized by ${\cal F}$.
If $X$ is $L$, $\SL{L}$, $\DaL{L}$, or $\DbL{L}$,
then 
$X$ contains all tautologies,
and $X$ is closed under tautological consequence;
therefore, $X$ can simulate {\bf LK} as below.
\begin{itemize}
\item
(Initial sequents) 
$\varphi \IMP \varphi, \varphi \IMP \BOT \in X$.
\item
(Inference rules)
$X$ is closed under the following rules. 
\[
\infer[\mbox{(weakening), where $\GA \subseteq \GA'$ and $\DE \subseteq \DE'$,}]{\bigwedge \GA' \IMP \bigvee \DE'}{\bigwedge \GA \IMP \bigvee \DE}
\]
\[
\infer[\mbox{($\IMP$left),}]{\bigwedge(\varphi \IMP \psi, \GA)\IMP \bigvee\DE}{
 \bigwedge\GA \IMP \bigvee(\DE, \varphi)
 &
 \bigwedge(\psi, \GA) \IMP \bigvee \DE
}
\quad
\infer[\mbox{($\IMP$right),}]{\bigwedge\GA \IMP \bigvee(\DE, \varphi \IMP \psi)}{
 \bigwedge(\varphi, \GA) \IMP \bigvee(\DE, \psi)
} 
\]
\[
\infer[\mbox{(cut).}]{\bigwedge(\GA,\PI)\IMP\bigvee(\DE,\SI)}{
  \bigwedge\GA \IMP \bigvee(\DE, \varphi)
  & 
  \bigwedge(\varphi, \PI) \IMP \bigvee\SI
}
\]
\end{itemize}
Using these,  
we generalize the arguments of Section~\ref{sec:cut-free-completeness}.
Roughly speaking, this is done 
by replacing terms 
according to Table~\ref{table:replace}.
In the following, we show a detailed proof of 
Theorem~\ref{th:generalization-S}.

First we redefine `\sS-saturation':
\begin{quote}
An S-sequent $\SSEQ{\PS^+}{\PH^+}$ is an
{\em \sS-saturation of $\SSEQ{\PS}{\PH}$}
$\quad\DEF\quad$
$\PS \subseteq \PS^+ \subseteq \SF{\PS,\PH}$; 
$\PH \subseteq \PH^+ \subseteq \SF{\PS,\PH}$; 
$\SSEQ{\PS^+}{\PH^+}$ is $\IMP$-saturated and $\Box$-left-saturated;
and 
$\bigwedge \PS^+ \IMP \bigvee\PH^+ \not\in\SL{L}$.
\end{quote}
Then we show the generalization of 
Lemma~\ref{lm:saturation}(2):
\begin{quote}
{\em 
If $\bigwedge \PS \IMP \bigvee \PH \not\in \SL{L}$,
then there is an \sS-saturation of $\SSEQ{\PS}{\PH}$.
}
\end{quote}
This is shown 
by simulating the proof of Lemma~\ref{lm:saturation}(2).
Note that the rule (\RuleLeftBox) is available in $\SL{L}$ as below
\[
\infer[(\RuleLeftBox)]
{\bigwedge(\Box\varphi, \GA) \IMP \bigvee\DE}{
  \bigwedge(\varphi, \GA) \IMP \bigvee\DE
}  
\]
because $\Box\varphi \IMP \varphi \in \SL{L}$.
Then we prove the generalization of Theorem~\ref{th:cut-free-completeness-S}:
\begin{quote}
{\em 
The following five conditions are equivalent.
\begin{itemize}
\item[{\rm (1)}]
$\bigwedge \PS \IMP \bigvee \PH \in \SL{L}$.

\item[{\rm (2)}]
$\bigwedge \PS \IMP \bigvee \PH$ is true 
at the bottom of any strongly constant $\lambda$-extension of any ${\cal F}$-model.

\item[{\rm (3)}]
$\bigwedge \PS \IMP \bigvee \PH$ is eventually always true in the tail of 
any strongly constant $\lambda$-extension of any ${\cal F}$-model.

\item[{\rm (4)}]
$\bigwedge \PS \IMP \bigvee \PH$ is eventually always true in the tail of 
any constant $\lambda$-extension of any ${\cal F}$-model.

\item[{\rm (5)}]
$\bigwedge \PS \IMP \bigvee \PH$ is eventually always true in the tail of 
any $\lambda$-extension of any ${\cal F}$-model.
\end{itemize}
}
\end{quote}
Note that $\varphi \in \SL{L}$ if and only if
$\bigwedge\emptyset \IMP \bigvee\{\varphi\} \in \SL{L}$;
therefore, showing the above equivalence
is sufficient for Theorem~\ref{th:generalization-S}.

(Proof of $(1)\Rightarrow (5)$)
We show that 
each element of $\SL{L}$ is eventually always true in the tail of 
any $\lambda$-extension of any ${\cal F}$-model.
This is easy because of 
the definition of $\SL{L}$, 
the assumption 
that $L$ is $\lambda$-tail-sound for ${\cal F}$,
and the fact that 
$\Box\varphi \IMP \varphi$
is eventually always true 
(this was shown in the proof of `$(1)\Rightarrow (6)$' of Theorem~\ref{th:cut-free-completeness-S}).

(Proof of $(5)\Rightarrow (4)\Rightarrow (3)$)
Trivial.

(Proof of $(3)\Rightarrow (2)$)
Shown in the proof of `$(4)\Rightarrow (3)$'
of Theorem~\ref{th:cut-free-completeness-S}.

(Proof of $(2)\Rightarrow (1)$)
We generalize the proof of `$(3)\Rightarrow (2)$'
of Theorem~\ref{th:cut-free-completeness-S}.
Suppose $\bigwedge \PS \IMP \bigvee \PH \not\in \SL{L}$.
By the `generalized Lemma~\ref{lm:saturation}(2)',
we get an \sS-saturation 
$\SSEQ{\PS^+}{\PH^+}$ of $\SSEQ{\PS}{\PH}$.
By the definition of \sS-saturation, 
we have 
$\bigwedge \PS^+ \IMP \bigvee\PH^+ \not\in\SL{L}$;
therefore
$\bigwedge \PS^+ \IMP \bigvee\PH^+ \not\in L$
because of the definition of $\SL{L}$.
Hence, there is an ${\cal F}$-model
$\Kmodel{W}{\R}{V}$ such that 
$\bigwedge \PS^+ \IMP \bigvee \PH^+$
is false at a world $t_0$
because of the assumption that $L$ is characterized by ${\cal F}$.
The rest of the proof is the same as in
the proof of `$(3)\Rightarrow (2)$'
of Theorem~\ref{th:cut-free-completeness-S}.

This completes the proof of Theorem~\ref{th:generalization-S}.
Similarly, 
Theorem~\ref{th:generalization-D} can be proved
by generalizing the proofs of
Theorems~\ref{th:cut-free-completeness-Da}
and 
\ref{th:cut-free-completeness-Db}.
\end{proof}

\begin{table}[h]
\caption{Rewriting for generalization}
\label{table:replace}
\begin{center}
\begin{tabular}{ll}
\hline
Terms in Section~\ref{sec:cut-free-completeness}
&
Replaced with
\\
\hline
GL-model & ${\cal F}$-model
\\
(cut-free) $\sGL \vdash \XSEQ{\GA}{\DE}$ & $\bigwedge \GA \IMP \bigvee \DE \in L $
\\
(cut-free) $\sS \vdash \XSEQ{\GA}{\DE}$ & $\bigwedge \GA \IMP \bigvee \DE \in \SL{L}$
\\
(analytic) $\sDa \vdash \XSEQ{\GA}{\DE}$ & $\bigwedge \GA \IMP \bigvee \DE \in \DaL{L}$
\\
(cut-free) $\sDb \vdash \XSEQ{\GA}{\DE}$ & $\bigwedge \GA \IMP \bigvee \DE \in \DbL{L}$
\\
Theorem~\ref{th:cut-free-completeness-GL}, 
Proposition~\ref{prop:trivial-GL}
&
assumption of the theorem 
\\
\hline
\end{tabular}
\end{center}
\end{table}

\begin{remark}
\textup{
In the proof of Theorem~\ref{th:Correspondence-D-Da},
we use the following particular properties of {\bf GL}:
$\Box\varphi \IMP \Box\Box\varphi \in {\bf GL}$, 
and 
${\bf GL} \subseteq
\DaL{\bf GL}$
(due to the property of the axioms of {\bf GL}
specified in the proof of Theorem~\ref{th:Correspondence-D-Da}).
Therefore, 
in general, 
the condition `$\varphi \in \DL{L}$' seems nonequivalent to
the conditions (1)-(4) of Theorem~\ref{th:generalization-D}.
}
\end{remark}

\medskip

In the rest of this section, 
we show an application of 
Theorem~\ref{th:generalization-D}
to the extensions of {\bf GL} and {\bf D} 
with linear order models.

Let $\lambda$ be a limit ordinal,
and $M$ be $\langle W, \R, V\rangle$.
We define `finite linear GL-model', `linear GL-model', and 
`$\lambda$-GL-model' as below.
\begin{quote}
$M$ is a {\em finite linear GL-model} 
$\DEF$
$\langle W, \R\rangle \simeq
\langle \{\alpha \mid \alpha \leq n\}, \ > \rangle$ 
for some natural number $n$.
(We call $\langle \{\alpha \mid \alpha \leq n\}, \ > \rangle$ 
a {\em finite linear GL-frame}.)

$M$ is a {\em linear GL-model} 
$\DEF$
$\langle W, \R\rangle \simeq
\langle \{\alpha \mid \alpha < \beta\}, \ > \rangle$ 
for some ordinal number $\beta > 0$.

$M$ is a {\em $\lambda$-GL-model} 
$\DEF$
$\langle W, \R\rangle \simeq
\langle \{\alpha \mid \alpha \leq \lambda\}, \ > \rangle$.
\end{quote}
Then we define logics 
{\GLLINF}, {\GLLIN}, $\DLINO{\lambda}$, and {\DLIN} as below.
\begin{quote}
${\GLLINF} =
\{\varphi \mid \mbox{$\varphi$ is true at any world of any 
finite linear GL-model.}\}$.

${\GLLIN} =
\{\varphi \mid \mbox{$\varphi$ is true at any world of any 
linear GL-model.}\}$.

$\DLINO{\lambda} = 
\{\varphi \mid \mbox{$\varphi$ is true at the bottom of 
any $\lambda$-GL-model.}\}$.

$\DLIN = 
\{\varphi \mid \mbox{$\varphi$ is true at the bottom of 
any $\lambda'$-GL-model, for any limit ordinal $\lambda'$.}\}$.
\end{quote}

\begin{theorem} 
\label{th:GLLIN}
${\GLLINF} = {\GLLIN}$.
(Therefore, {\GLLIN} is characterized by 
the class of finite linear GL-frames.)
\end{theorem}
\begin{proof}
{\GLLINF} is known to be axiomatized by 
adding an axiom scheme 
$\Box(\Box\varphi \IMP \psi) \OR \Box(\psi \AND \Box\psi \IMP \varphi)$
to {\HGL} (\S 25 in \cite{Gabbay_Investigations}).
We call this system {\HGLLIN}.
The soundness of {\HGLLIN} with respect to {\GLLIN} is easily shown.
Therefore ${\GLLINF} \subseteq {\GLLIN}$ is shown as below:
$
\varphi \in {\GLLINF} 
\Longrightarrow
{\HGLLIN}\vdash\varphi
\Longrightarrow
\varphi \in {\GLLIN}
$. 
The converse inclusion ${\GLLINF} \supseteq {\GLLIN}$ 
is trivial by the definition.
\end{proof}

\begin{theorem}
\label{th:application}
Let $\lambda$ be a limit ordinal.
The following conditions are equivalent.
\begin{itemize}
\item[{\rm (1)}]
$\varphi \in \DaL{\GLLIN}$.

\item[{\rm (2)}]
$\varphi \in \DbL{\GLLIN}$.

\item[{\rm (3)}]
$\varphi$ is true 
at the bottom of any $\lambda$-extension of any 
finite linear GL-model.

\item[{\rm (4)}]
$\varphi$ is true 
at the bottom of any $\lambda'$-extension of any 
finite linear GL-model,
for any limit ordinal $\lambda'$.

\item[{\rm (5)}]
$\varphi \in \DLINO{\lambda}$.

\item[{\rm (6)}]
$\varphi \in \DLIN$.

\end{itemize}

\end{theorem}

\begin{proof}
If $M$ is a $\lambda$-extension of a finite linear GL-model,
and if $M'$ is a submodel of $M$ as in Figure~\ref{Fig:GLLIN},
then $M'$ is a $\lambda$-GL-model
and the truth value of each formula at each world 
is inherited from $M$.
Using this, 
we can show the following three facts.
(i)
{\GLLIN} is $\lambda$-tail-sound for the class of
finite linear GL-frames.
(ii) 
The conditions (3) and (5) are equivalent.
(iii) 
The conditions (4) and (6) are equivalent.
Then, the fact (i) and 
Theorems~\ref{th:GLLIN} and \ref{th:generalization-D}
imply the equivalence between 
the conditions (1), (2), and (3).
Finally, 
`$(3) \Longleftrightarrow (4)$'
is shown similarly  to Remark~\ref{rem:bound-lambda}.
\end{proof}

\begin{figure}[h]

\begin{minipage}{6cm}
\begin{tikzpicture} [x=1mm,y=1mm]

\node at (0,0) {$\bullet$};
\draw [->,thick] (0,-3)--(0,-1);
\node at (0,-4) {$\bullet$};
\draw [->,thick] (0,-7)--(0,-5);
\node at (0,-8) {$\cdot$};
\node at (0,-9) {$\cdot$};
\node at (0,-10) {$\cdot$};
\draw [->,thick] (0,-13)--(0,-11);
\node at (0,-14) {$\bullet$};
\draw [->,thick] (0,-17)--(0,-15);
\node at (0,-18) {$\cdot$};
\node at (0,-19) {$\cdot$};
\node at (0,-20) {$\cdot$};
\draw [->,thick] (0,-23)--(0,-21);
\node at (0,-24) {$\bullet$};
\draw [->,thick] (5,-16) to [out=120, in=0] (1,-14);
\node at (5,-17) {$\bullet$};
\draw [->,thick] (5,-20)--(5,-18);
\node at (5,-21) {$\bullet$};
\draw [->,thick] (5,-24)--(5,-22);
\node at (5,-25) {$\cdot$};
\node at (5,-26) {$\cdot$};
\node at (5,-27) {$\cdot$};
\node at (5,-28) {$\cdot$};
\node at (5,-29) {$\cdot$};
\node at (5,-30) {$\cdot$};
\node at (5,-32) {$\bullet$};
\node [right] at (6,-32) {bottom};

\node at (2,-35) {$M$};

\node at (-50,0) {};
\end{tikzpicture} 
\end{minipage}
\begin{minipage}{6cm}
\begin{tikzpicture} [x=1mm,y=1mm]

\node at (0,0) {$\bullet$};
\draw [->,thick] (0,-3)--(0,-1);
\node at (0,-4) {$\bullet$};
\draw [->,thick] (0,-7)--(0,-5);
\node at (0,-8) {$\cdot$};
\node at (0,-9) {$\cdot$};
\node at (0,-10) {$\cdot$};
\draw [->,thick] (0,-13)--(0,-11);
\node at (0,-14) {$\bullet$};
\draw [->,thick] (5,-16) to [out=120, in=0] (1,-14);
\node at (5,-17) {$\bullet$};
\draw [->,thick] (5,-20)--(5,-18);
\node at (5,-21) {$\bullet$};
\draw [->,thick] (5,-24)--(5,-22);
\node at (5,-25) {$\cdot$};
\node at (5,-26) {$\cdot$};
\node at (5,-27) {$\cdot$};
\node at (5,-28) {$\cdot$};
\node at (5,-29) {$\cdot$};
\node at (5,-30) {$\cdot$};
\node at (5,-32) {$\bullet$};
\node [right] at (6,-32) {bottom};

\node at (2,-35) {$M'$};

\node at (-50,0) {};
\end{tikzpicture} 
\end{minipage}
\caption{A $\lambda$-extension of a finite linear GL-model
and its submodel.
}
\label{Fig:GLLIN}
\end{figure}

\begin{remark}
\textup{
Also the condition `$\varphi \in \DL{\GLLIN}$'
is equivalent to the conditions in 
Theorem~\ref{th:application}
because
we can show $\DaL{\GLLIN} = \DL{\GLLIN}$
in the same way as Theorem~\ref{th:Correspondence-D-Da},
using the Hilbert-style system {\HGLLIN} 
(mentioned in the proof of Theorem~\ref{th:GLLIN}).
A cut-free sequent calculus for {\GLLIN}
was  discussed in 
\cite{Valentini86}.
}
\end{remark}

\

\noindent
{\large {\bf Acknowledgements}}
\\
The authors would like to thank the referees for their 
valuable comments.
This work was supported by JSPS KAKENHI Grant Numbers JP20K03712
and JP23K03200.

\bibliographystyle {plain}
\bibliography{Myreference}

\end{document}